\documentclass[a4paper,8pt,oneside,bookmarks]{article}
\usepackage[utf8x]{inputenc}
\usepackage{graphicx}

\usepackage{wrapfig}
\usepackage[pdftex,bookmarks]{hyperref}
\hypersetup{
pdftitle={Stability properties on locally $L^0$-convex modules and conditional Lebesgue property for conditional convex risk measures},
pdfauthor={José M. Zapata}
}

\usepackage[english]{babel}
\usepackage{amsmath}
\usepackage{amsthm}
\usepackage{amsfonts}

\usepackage{mathrsfs} 

\usepackage{amscd}

\usepackage{epigraph}

\usepackage{amssymb}

\usepackage{relsize}

\newcommand{\N}{\mathbb{N}}
\newcommand{\R}{\mathbb{R}}
\newcommand{\E}{\mathbb{E}}
\newcommand{\PP}{\mathbb{P}}
\newcommand{\Q}{\mathbb{Q}}
\newcommand{\F}{\mathcal{F}}
\newcommand{\EE}{\mathcal{E}}
\newcommand{\A}{\mathcal{A}}

\DeclareMathOperator*{\esssup}{ess.sup\,}
\DeclareMathOperator*{\essinf}{ess.inf\,}
\DeclareMathOperator*{\esslimsup}{ess.limsup\, }
\DeclareMathOperator*{\essliminf}{ess.liminf\, }

\DeclareMathOperator{\spa}{\textbf{span}}

\DeclareMathOperator{\dom}{\textbf{dom}}

\DeclareMathOperator{\epi}{\textbf{epi}}
\DeclareMathOperator{\nint}{\textbf{int}}
\DeclareMathOperator{\nVert}{\boldsymbol\Vert}

\DeclareMathOperator{\nmax}{\textbf{max}\,}

\DeclareMathOperator{\ninf}{\textbf{inf}\,}
\DeclareMathOperator{\nsup}{\textbf{sup}\,}

\DeclareMathOperator{\nco}{\textbf{co}}
\DeclareMathOperator*{\nlimsup}{\textbf{limsup}\, }
\DeclareMathOperator*{\nliminf}{\textbf{liminf}\, }
\DeclareMathOperator*{\nlim}{\textbf{lim}\, }

\newtheorem{defn}{Definition}[section]
\newtheorem{rem}{Remark}[section]
\newtheorem{prop}{Proposition}[section]
\newtheorem{lem}{Lemma}[section]
\newtheorem{thm}{Theorem}[section]

\newtheorem{example}{Example}[section]

\providecommand{\keywords}[1]{\textbf{\textit{Keywords:}} #1}

\begin{document}


\title{Stability in locally $L^0$-convex modules and a conditional version of James' compactness theorem}

\author{José Orihuela \thanks{Universidad de Murcia, Dpto. Matemáticas, 30100 Espinardo, Murcia, Spain, e-mail: joseori@um.es}, José M. Zapata \thanks{Universidad de Murcia, Dpto. Matemáticas, 30100 Espinardo, Murcia, Spain, e-mail: jmzg1@um.es. } \thanks{Second author was partially supported by the grant MINECO MTM2014-57838-C2-1-P}}

\maketitle

\begin{abstract}

Locally $L^0$-convex modules were introduced in [D. Filipovic, M. Kupper, N. Vogelpoth. Separation and duality in locally $L^0$-convex modules. J. Funct. Anal. 256(12), 3996-4029 (2009)] as the analytic basis for the study of conditional risk measures. Later,  the algebra of conditional sets was introduced in [S. Drapeau, A. Jamneshan, M. Karliczek, M. Kupper. The algebra of conditional sets and the concepts of conditional topology and compactness. J. Math. Anal. Appl. 437(1), 561-589 (2016)]. In this paper we study locally $L^0$-convex modules, and find exactly which subclass of locally $L^0$-convex modules can be identified with the class of locally convex vector spaces within the context of conditional set theory. Second, we provide a version of the classical James' theorem of characterization of weak compactness for conditional Banach spaces. Finally, we state a conditional version of the Fatou and Lebesgue properties for conditional convex risk measures and, as application of the developed theory, we stablish a version of the so-called Jouini-Schachermayer-Touzi theorem for robust representation of conditional convex risk measures defined on a $L^\infty$-type module.




\end{abstract}

\keywords{stability properties; locally $L^0$-convex module; conditionally locally convex space;  James' compactness theorem;  conditional convex risk measure; conditional Lebesgue property;  Jouini-Schachermayer-Touzi Theorem}


\section*{Introduction}
\label{intro}

The study of risk measures (or monetary utility functions, i.e. the negative value of a risk measure) was initiated by Artzner et al.\cite{key-24}, by defining and studying the concept of coherent risk measure. Föllmer and Schied \cite{key-25} and, independently, Frittelli and Gianin \cite{key-26} later introduced  the more general concept of convex risk measure. Both kinds of risk measures are defined in a static setting, in which only two instants of time matter, today $0$ and tomorrow $T$; and the analytic framework used is the classical convex analysis, which perfectly applies in this simple model cf.\cite{key-33,key-28,key-27}. For instance, Delbaen \cite{key-35} in the coherent case and later Föllmer et al.\cite{key-36}  in the general convex case, obtained that any convex risk measure $\rho$ defined on $L^\infty(\Omega,\mathcal{F},\PP)$ has a representation formula as follows
\[
\begin{array}{cc}
\rho(x)=\sup \left\{\E_Q[-x] - \alpha(Q)\:;\: Q\ll\PP\right\} & \textnormal{ for all }x\in L^\infty
\end{array}
\] 
 if, and only if, $\rho$ is order lower semicontinuous  —equivalently, $\rho$ has the Fatou property—. Moreover, the so-called Jouini-Schachermayer-Touzi theorem \cite[Theorem 2]{key-33} (see also \cite[Theorem 5.2]{key-34} for the original reference) states that the representation formula is attained  —i.e, the supremum turns out to be a maximum— for all $x\in L^\infty$ if, and only if, $\rho$ is order continuous —equivalently, $\rho$ has the Lebesgue property—, which is also equivalent to the weak compactness of the sublevel sets of $\alpha$.  

However, when dealing a dynamic or multiperiod setting, in which the arrival of new information at an intermediate time $0<t<T$ can be taken into account, it is quite delicate to apply convex analysis, as Filipovic et al.\cite{key-10} explained. In order to overcome these difficulties Filipovic et al.\cite{key-10} proposed to consider a modular framework, where scalars are random variables instead of real numbers. Namely, they considered modules over $L^0(\Omega,\mathcal{F},\PP)$ the ordered ring (of equivalence classes) of $\F$-measurable random variables, where $\F$ is a $\sigma$-algebra that models the market information available at some time $t$. For this purpose, they established the concept of locally $L^0$-convex module and proved randomized versions of some important theorems from convex analysis.  


Since then, the theory of locally $L^0$-convex has been applied to the study of conditional risk measures. Namely, Filipovic et al.\cite{key-30} used the so-called  $L^p$-type modules as a model space. For a given probability space $(\Omega,\EE,\PP)$, a sub-$\sigma$-algebra $\F$ and $1\leq p \leq \infty$, the $L^p$-type module, denoted by $L^p_\F(\EE)$, is defined as the smallest $L^0(\F)$-submodule of $L^0(\EE)$ containing the space $L^p(\EE)$ of measurable functions, i.e. $L^p_\F(\EE)=L^0(\F)L^p(\EE)$.   They considered conditional $L^0(\F)$-convex risk measures as a $L^0(\F)$-convex, cash-invariant and monotone function from $L^p_\F(\EE)$ to $L^0(\F)$. 


Recently, S. Drapeau et al.\cite{key-7}, in a more abstract level than $L^0$-theory, created a new framework, called the algebra of conditional sets, in which stability properties are supposed on all structures and the techniques developed in the $L^0$-theory can be applied in a structured way. They also introduced the notion of conditional real line and showed its relation with $L^0$. Then, they succeeded in constructing a conditional topology and a conditional real analysis, proving conditional versions of some classical theorems of topology and functional analysis  in this framework. 


In the first part of this paper, we look deeper into the connection between locally $L^0$-convex modules and the  conditional set theory. Namely, for fixed the measure algebra associated to an underlying probability space, we find, in terms of a equivalence of categories, which class of locally $L^0$-convex modules can be identified with the class of conditionally locally convex vector spaces. This will allow us to apply the machinery of conditional set theory to locally $L^0$-convex modules and, conversely, to draw from $L^0$-theory some existing important results  (for instance, theorems from \cite{key-10,key-5,key-6,key-13}) to conditional set theory when the underlying boolean algebra is a measure algebra. Also, in a negative direction, we show examples of locally $L^0$-convex that lack stability on either the algebraic structure or the topological structure, highlighting that some locally $L^0$-convex modules fall outside the scope of conditional set theory. For instance, we find that the weak topologies that typically have been employed for topological $L^0$-modules are not necessarily stable, and we need to consider a finer version of they in order to allow the conditional set approach.      

 In the recent literature about dynamic or conditional risk measures, some results on robust representation  have been obtained cf.\cite{key-20,key-42,key-30,key-37}. For instance, Deflefsen and Scandolo \cite{key-20} showed that, for a given   probability space $(\Omega,\F,\PP)$ and a sub-$\sigma$-algebra $\F$ of $\EE$, if $\rho:L^\infty(\EE)\rightarrow L^\infty(\F)$ is a risk measure which is $L^\infty(\F)$-cash invariant and $L^0(\F)$-convex, then it has the Fatou property if, and only if, it can be represented as follows
\begin{equation}
\label{eq: FatouDS}
\begin{array}{cc}
\rho(x)=\esssup\left\{ \E_Q[-x|\F] - \alpha(Q) \:;\: Q\ll\PP,\: Q|_\F=\PP|_\F\right\} & \textnormal{ for }x\in L^\infty(\EE).
\end{array}
\end{equation}

However, the attainability of these representations in terms of some compactness condition has not been studied so far. Therefore, another contribution of the present paper is to state a suitable Lebesgue property for the dynamic case, and provide a version of the mentioned Jouini-Schachermayer-Touzi theorem for conditional convex risk measures defined on a $L^\infty$-type module. In this way, by using scalarization techniques and the machinery of conditional set theory, we characterize the robustness of a conditional risk measure in terms of the conditional Lebesgue property and also in terms  of conditional compactness. As a consequence, we also provide a result showing that, for each $x\in L^\infty(\EE)$, the supremum (\ref{eq: FatouDS}) turns out to be a maximum if, and only if, some conditional compactness condition is fulfilled. 

The proof of the Jouini-Schachermayer-Touzi theorem is typically based on versions of original version of the classical James’ theorem for weak compactness cf.\cite{key-33,key-34}. Thus, another contribution, which will be a key piece in the proof of the version of Jouini-Schachermayer-Touzi theorem provided in this paper, is a perturbed version of James’ compactness theorem in the framework
of conditional Banach spaces. As a particular case, we will also obtain a non perturbed version of this theorem in this framework.

This paper is structured as follow. In Section 1, we first study stability properties on locally $L^0$-convex modules, and collect some results and examples exhibiting how the different types of stability properties affect the algebraic and topological structures of the locally $L^0$-convex modules; and second, we recall the framework of conditional sets —this setting will be use in the remainder of this work—, and show how locally $L^0$-modules are connected with conditional sets. In Section 2, we obtain a conditional version of the classical  James' compactness theorem; and, as application,  we define conditional versions of the Fatou and Lebesgue properties, and prove a version of Jouini-Schachermayer-Touzi theorem for conditional convex risk measures on a $L^\infty$-type module.



\section{Locally $L^0$-convex modules and conditionally locally convex vector spaces}

%
%


\subsection{Locally $L^0$-convex modules and stability properties}

First and for the convenience of the reader, let us give some notation. Let $\left(\Omega,\mathcal{F},\PP\right)$ be a given probability space, and let us consider $L^{0} \left(\Omega,\mathcal{F},\PP\right)$, or simply  $L^{0}$, the set of equivalence classes of  real valued $\mathcal{F}$-measurable random variables. It is known that the triple $\left(L^{0},+,\cdot\right)$ endowed with the partial order of  almost sure dominance is a lattice ordered ring. We will follow the common practice of identifying a random variable with its equivalence class.

Given $\eta,\xi\in L^0$,  we will write $\eta\geq \xi$ if $\PP\left( \eta\geq\xi\right)=1$, and $\eta>\xi$ if $\PP\left( \eta> \xi \right)=1$. 
We also define $L_{+}^{0}:=\left\{ \eta\in L^{0}\:;\: \eta\geq 0\right\}$ and $L_{++}^{0}:=\left\{ \eta\in L^{0}\:;\: \eta>0  \right\}$. 
 We will denote by $\bar{L^{0}}$, the set of equivalence classes of  $\mathcal{F}$-measurable random variables taking values in $\overline{\R}=\R\cup\{\pm\infty\}$. The partial order of  almost sure dominance is extended to $\bar{L^{0}}$ in a natural way. Furthermore, given a subset $H\subset L^{0}$, then $H$ has both an infimum and a supremum in $\bar{L^{0}}$ for the order of  almost sure dominance that will be denoted by $\essinf H$ and $\esssup H$, respectively. 

This order also allows us to define a topology. We define $B_{\varepsilon}:=\left\{ \eta\in L^{0}\:;\:\left|\eta\right|\leq\varepsilon\right\}$ the ball of radius $\varepsilon\in L_{++}^{0}$ centered at $0\in L^{0}$. Then, $\mathscr{U}:=\left\{ \eta+B_{\varepsilon}\:;\:\eta\in L^0,\:\varepsilon\in L_{++}^{0}\right\} $ is a neighborhood base of a Hausdorff topology on $L^{0}$  (see Filipovic et al.\cite{key-10}). 

We also define the measure algebra associated to $\F$, denoted by $\A_\F$ —or simply $\A$—, obtained by identifying two events of $\F$ if, and only if, their symmetric difference is $\PP$-negligible. We will denote by $a,b,...$ the elements of $\A$, and by  $0$ and $1$ the equivalence classes of $\emptyset$ and $\Omega$, respectively.\footnote[1]{This notation is used in accordance with \cite{key-7}.} In this way, we obtain a complete Boolean algebra $(\A,\vee,\wedge,1,0)$. From time to time we will identify an element of $\A$  with some representative of $\F$.  

Given $a\in\A$, where $a$ is the equivalence class of some $A\in\F$, we define $1_a$ as the equivalence class in $L^0$ of the characteristic function $1_A$ (note that this definition does not depend on the representative $A$). 

We also define the set of partitions of $a$ given by $p(a):=\{ \{a_k\}_{k\in \N} \subset \A \:;\: \vee a_k=a\textnormal{, }a_i \wedge a_j = 0, \textnormal{ for all }i\neq j\textnormal{, }i,j\in \N  \}$.  Note that we allow $a_k=0$ for some $k\in\N$. 

We also denote by $\A_a:=\{b\in\A\:;\: b\leq a \}$ the trace of $\A$ on $a$, which is also a complete Boolean algebra and can be identified with the measure algebra associated to the probability space $(A,\F_A,\PP(\cdot|A))$ where $A$ is a representative of $a$ and $\F_A:=\{B\in\F\:;\: B\subset A \}$.

Let us recall some notions of the theory of locally $L^0$-convex modules:

\begin{itemize}
	\item 

\label{defn: L0module}\cite[Definition 2.1]{key-10}
A topological $L^{0}$-module $E\left[\mathscr{T}\right]$ is a $L^{0}$-module $E$ endowed with a topology $\mathscr{T}$ such that 
\[
\begin{array}{cc}
E\left[\mathscr{T}\right]\times E\left[\mathscr{T}\right]\longrightarrow E\left[\mathscr{T}\right],\left(x,x'\right)\mapsto x+x', & L^{0}\left[\left|\cdot\right|\right]\times E\left[\mathscr{T}\right]\longrightarrow E\left[\mathscr{T}\right],\left(\eta,x\right)\mapsto {\eta}x
\end{array}
\]
are continuous with the corresponding product topologies.

\item
\label{defn: conveL0mod}
\cite[Definition 2.2]{key-10}
A topology $\mathscr{T}$ on a $L^{0}$-module $E$ is said to be locally $L^{0}$-convex 
 if there is a neighborhood base $\mathcal{U}$ of $0\in{E}$  such that each $U\in\mathcal{U}$ is
\begin{enumerate}
\item $L^{0}$-convex, i.e. ${\eta}x+{(1-\eta)}y\in U$ for all $x,y\in U$
and $\eta\in L^{0}$ with $0\leq \eta\leq 1$;
\item $L^{0}$-absorbent, i.e. for all $x\in E$ there is a $\eta\in L_{++}^{0}$
such that $x\in {\eta}U$;
\item $L^{0}$-balanced, i.e. ${\eta}x\in U$ for all $x\in U$ and $\eta\in L^{0}$
with $\left|\eta\right|\leq 1$.
\end{enumerate}
In this case, $E\left[\mathscr{T}\right]$ is called a locally $L^0$-convex module. 

\item
\label{defn: seminorm}
\cite[definition 2.3]{key-10}
A function $\left\Vert \cdot\right\Vert :E\rightarrow L_{+}^{0}$
is a $L^{0}$-seminorm on $E$ if:
\begin{enumerate}
\item $\left\Vert {\eta}x\right\Vert =\left|\eta\right|\left\Vert x\right\Vert $
for all $\eta\in L^{0}$ and $x\in E$;
\item $\left\Vert x+y\right\Vert \leq\left\Vert x\right\Vert +\left\Vert y\right\Vert$, 
for all $x,y\in E.$
\end{enumerate}
If moreover, $\left\Vert x\right\Vert=0$ implies $x=0$, 
then $\left\Vert \cdot\right\Vert $ is a $L^{0}$-norm on $E$.
\end{itemize}

Let $\mathscr{P}$ be a family of $L^0$-seminorms on a $L^{0}$-module
$E$. Given  a finite subset $F$ of $\mathscr{P}$ and $\varepsilon\in L_{++}^{0}$, 
we define 

\[
\begin{array}{cc}
U_{F,\varepsilon}:=\left\{ x\in E\:;\:\left\Vert x\right\Vert_F \leq\varepsilon\right\}, & \textnormal{ where }\left\Vert x\right\Vert_F:=\esssup\{\Vert x\Vert \:;\: \Vert\cdot\Vert\in F \}.
\end{array}
\]

Then $\mathscr{U}:=\left\{ U_{F,\varepsilon}\:;\:\varepsilon\in L_{++}^{0},\: F\subset \mathscr{P}\textnormal{ finite} \right\} $ is a neighborhood base of $0\in E$ for some locally $L^0$-convex topology $\mathcal{T}$, which is called the topology induced by $\mathscr{P}$ (see \cite{key-10}). $E$ endowed with this topology is denoted by $E\left[\mathscr{P}\right]$.


Let us recall more notions:

\begin{itemize}
	\item Given a sequence $\{x_k\}$ in a $L^0$-module $E$ and a partition $\left\{ a_k\right\}\in p(1)$, an element $x\in E$ is said to be a \textit{concatenation} of $\left\{ x_{k}\right\}$ along $\left\{ a_k\right\}\in p(1)$ if $1_{a_k}x_k=1_{a_k}x$ for all $k\in\N$.
\item Let $E$ be a $L^0$-module, a non-empty subset $K\subset E$ is said to be \textit{stable}, if for each sequence $\left\{ x_{k}\right\}$ in $K$ and each partition $\left\{ a_{k}\right\}\in p(1)$, there exists a unique concatenation $x\in E$ of $\{x_k\}$ along $\{a_k\}$.
\end{itemize}

It is important to highlight that there are examples of $L^0$-modules such that, for $\{x_k\}\subset E$ and $\{a_k\}\in p(1)$, there can be more than one concatenation of $\{x_k\}$ along $\{a_k\}$ as \cite[Example 1.1]{key-15} exhibits. However, if $E$ is stable, then for each sequence $\{x_k\}\subset E$ and every partition $\{a_k\}\in p(1)$ there exists a unique concatenation $x\in E$ of $\{x_k\}$ along $\{a_k\}$, and we will use the notation $x=\sum_{k\in\N} 1_{a_k}x_k$. Also, it is important to note that not every $L^0$-module is stable, even if it has uniqueness for concatenations; for instance, see \cite[Example 2.12]{key-10}.  

%
%

For any family of $L^0$-seminorms, one can also define a topology by using another method, which has been treated in the literature under  different approaches (see \cite[Definition 2.20]{key-10} and comments after Definition 2.5 and after Remark 3.5 of \cite{key-6}):

\begin{defn}
Let $\mathscr{P}$ be a family of $L^0$-seminorms on a $L^{0}$-module
$E$. Given a partition $\{a_k\}\in p(1)$,  a family $\{F_k\}_{k\in\N}$ of non-empty finite subsets of $\mathscr{P}$, and $\varepsilon\in L_{++}^{0}$,  we define 

\[
\begin{array}{cc}
U_{\{F_k\},\{a_k\},\varepsilon}:=\left\{ x\in E\:;\: \sum 1_{a_k} \left\Vert x\right\Vert_{F_k}\leq\varepsilon\right\} & \textnormal{ with }\left\Vert x\right\Vert_{F_k}:=\esssup\{\Vert x\Vert \:;\: \Vert\cdot\Vert\in F_k \}.
\end{array}
\]
Then 
\[
\mathscr{U}:=\left\{U_{\{F_k\},\{a_k\},\varepsilon}\:;\:\varepsilon\in L_{++}^{0},\: F_k\subset \mathscr{P} \textnormal{ finite for all }k\in\N,\{a_k\}\in p(1) \right\}
\]
is a neighborhood base of $0\in E$ for some locally $L^0$-convex topology on $E$, which is finer than the topology induced by $\mathscr{P}$. This topology will be referred to as \emph{the topology stably induced} by $\mathscr{P}$, and $E$, endowed with this topology, will be denoted by $E\left[\mathscr{P}_{cc}\right]$.
\end{defn}

 %


We have the following result:

\begin{thm}
\label{thm: caracterizacionII}
Let $E\left[\mathscr{T}\right]$ be a stable topological $L^{0}$-module. Then $\mathscr{T}$ is stably induced by a family of $L^0$-seminorms if, and only if, there is a neighborhood base $\mathscr{U}$ of $0\in{E}$ for which 
\begin{enumerate}
\item each $U\in\mathscr{U}$ is $L^{0}$-convex, $L^{0}$-absorbent, $L^{0}$-balanced and stable;
\item and for each $\{a_k\}\in p(1)$ and $\{U_k\}\subset \mathscr{U}$ it holds $\sum 1_{a_k}U_k\in\mathscr{U}$.
\end{enumerate}

In this case, $\mathscr{T}$ is \textit{stable}; that is, for every $\left\{ a_{k}\right\}\in p(1)$ and every countable family of non-empty open sets $\left\{O_{k}\right\}$, it holds that $\sum 1_{a_k} O_k$ is again an open set. 
\end{thm}
\begin{proof} 
If $\mathscr{T}$ is stably induced by a family of $L^0$-seminorms $\mathscr{P}$, then the family
\[
\mathscr{U}:=\left\{U_{\{F_k\},\{a_k\},\varepsilon}\:;\:\varepsilon\in L_{++}^{0},\: F_k\subset \mathscr{P} \textnormal{ finite for all }k\in\N,\{a_k\}\in p(1) \right\}
\]
is a neighborhood base of $0\in E$ with respect to the topology $\mathscr{T}$, which satisfies the conditions 1 and 2 of the statement. 

Conversely, let $\mathscr{U}$ be a neighborhood base of $0\in E$ satisfying 1 and 2 above.  From \cite[Theorem 2.6]{key-17}, we know that 1 implies that $\mathscr{T}$ is induced by the family of $L^0$-seminorms\footnote[2]{Theorem 2.6 of \cite{key-17}  considers a mild stability property for the elements of the neighborhood base considered in the theorem, namely being closed under countable concatenations. It is easy to show that, when $E$ is stable, a subset $U$ with this property turns out to be the stable. Further, in \cite{key-17}, and independently in \cite{key-18}, a counterexample was provided showing that this extra condition cannot be removed if one wants $\mathscr{T}$ to be induced by a family of $L^0$-seminorms.} $\{p_U\}_{U\in\mathscr{U}}$, where $p_U:E\rightarrow L^0_+$ is the gauge function (see \cite[Definition 2.21]{key-7}). Let us show that, in fact, $\mathscr{T}$ is stably induced by that family. Indeed, let us fix a partition $\{a_k\}\in p(1)$, a sequence $\{F_k\}$ of finite subsets of $\mathscr{U}$ and $\varepsilon\in L^0_{++}$. For each $k\in\N$, let us choose $U_k\in \mathscr{U}$ with $U_k\subset \cap F_k$. Then, inspection shows that 
\[
\left\{ x\in E \:;\: \sum 1_{a_k}\esssup_{U\in F_k} p_U(x)\leq \varepsilon  \right\}\supset\left\{ x\in E \:;\: \sum 1_{a_k} p_{U_k}(x)\leq \frac{\varepsilon}{2}  \right\}= 
\]
\[
=\left\{ x\in E \:;\: p_{\sum 1_{a_k} U_k}(x)\leq \frac{\varepsilon}{2} \right\}.
\]
Since $\sum 1_{a_k}U_k\in\mathscr{U}$, the result follows. 

Finally, given $\left\{ a_{k}\right\}\in p(1)$ and $\left\{O_{k}\right\}$ a countable family of non-empty open sets, let us show that $O:=\sum 1_{a_k}O_k$ is open. Indeed, for a fixed $x\in O$, let $\mathscr{U}$ be a neighborhood base of $0$ as in Theorem \ref{thm: caracterizacionII}. Also, let $\{x_k\}$ be so that $x_k\in O_k$ and $1_{a_k}x_k=1_{a_k}x$ for all $k$. Then, for each $k$, we can choose $U_k\in\mathscr{U}$ with $x_k + U_k\subset O_k$. Therefore $x + \sum 1_{a_k}U_k\subset \sum 1_{a_k}O_k=O$.

\end{proof}



\begin{defn}
A stable topological $L^{0}$-module $E\left[\mathscr{T}\right]$ in the conditions of Theorem \ref{thm: caracterizacionII} is called \textit{locally stable $L^0$-convex module}.
\end{defn}

In the following example, we give a locally $L^0$-convex topology which is induced by a family of $L^0$-seminorms and yet it is not stably induced by any family of $L^0$-seminorms.

\begin{example}
\label{ex3}
Let $(\Omega,\mathcal{F},\PP)$ be an atomless probability space and $\left\{ a_{k}\right\}\in p(1)$ with $a_{k}\neq 0$ for each $k\in\N$. Let us take the $L^0$-module $(L^0)^{\N}$. For each $k\in\N$, let us consider the application $p_{k}(x_n):=|x_k|$ with $(x_n)\in(L^0)^{\N}$. Then $\{p_k\:;\: k\in\N \}$ is a family of $L^0$-seminorms which induces the product topology on $(L^0)^{\N}$. However,  it is not stably induced by a family of $L^0$-seminorms. Indeed, let us define $O_1:=(0,1)\times(L^0)^{\N}$, and for each $n>1$, let us put $O_n:=(L^0)^{n-1}\times(0,1)\times(L^0)^{\N}$. Then $\sum 1_{a_k}O_k=\prod_{k\in\N} 1_{a_k}(0,1)+1_{a_k^c}L^0$ is not an open subset. In view of  Theorem \ref{thm: caracterizacionII}, the product topology cannot be stably induced by any family of $L^0$-seminorms.   
\end{example}

Filipovic et al.\cite{key-10} introduced the topological dual of a topological $E[\mathscr{T}]$ $L^0$-module $E$, which is denoted by 
\[
E[\mathscr{T}]^*=E^*=\left\{ \mu :E\rightarrow L^0\:;\:\mu\textnormal{ is }L^0-\textnormal{linear and continuous}\right\}. 
\]

Let $E[\mathscr{T}]$ be a locally $L^0$-convex module. Let us consider the family of $L^0$-seminorms $\{ q_{x^*}\}_{x^*\in E^*}$ defined by $q_{x^*}(x):=|x^*(x)|$ for $x\in E$. Then, we can endow $E$ with the weak topology $E[\sigma(E,E^*)]$ and with the stable weak topology $E[\sigma(E,E^*)_{cc}]$. Analogously, we have the weak-$*$ and the stable weak-$*$ topology.  

The stable weak (resp. weak-$*$) topology is finer than the weak (resp. weak-$*$) topology. The following example shows that both are not necessarily equal, even when $E$ is stable: 

\begin{example}
\label{ex4}
Filipovic et al.\cite{key-10} introduced the following locally $L^0$-convex modules, which are called $L^p$-type modules. Namely, let $(\Omega,\EE,\PP)$ be a probability space such that $\F$ is a sub-$\sigma$-algebra of $\EE$ and $p\in [1,+\infty]$. Then we can define the $L^0(\F)$-module  
$L^p_\F(\EE):=L^0(\F) L^p(\EE)$, which is stable (see \cite[Proposition 3.3]{key-37}), and for which
\[
\left\Vert x | \mathcal{F} \right\Vert _{p}:=\begin{cases}
\E_\PP\left[\left|x\right|^{p}|\mathcal{F}\right]^{1/p} & \textnormal{ if } p<\infty\\
\essinf\left\{ y\in\bar{L}^{0}\left(\mathcal{F}\right)\:;\: y\geq\left|x\right|\right\}  & \textnormal{ if } p=\infty
\end{cases}
\]
defines a $L^0(\F)$-norm.

Besides, it is known that for $1\leq p<+\infty$, if $1<q\leq+\infty$ with $1/p+1/q=1$, the map $T:L^p_\F(\EE)\rightarrow [L^q_\F(\EE)]^*$, $y\mapsto T_y$ defined by $T_y(x):=\E_\PP[x y|\F]$ is a $L^0(\F)$-isometric isomorphism (see \cite[Theorem 4.5]{key-5}).

The weak topologies are defined, and the family of sets 
\[
U_{F,\varepsilon}:=\left\{x\in L^p_\F(\EE)\:;\: |\E_\PP[x y|\F]|\leq\varepsilon,\:\forall y\in F\right\},
\]
where $F$ is a finite subset of $L^q_\F(\EE)$ and $\varepsilon\in L^0_{++}(\F)$, constitutes a neighborhood base of $0\in L^p_\F(\EE)$ for the weak topology.

Let us consider the particular case:  $\Omega=(0,1)$, $\mathcal{E}=\mathcal{B}(\Omega)$ the Borel $\sigma$-algebra, $A_k=[\frac{1}{2^k},\frac{1}{2^{k-1}})$ with $k\in\mathbb{N}$, $\F:=\sigma(\{A_k\:;\:k\in\N\})$ and $\PP:=\lambda$ the Lebesgue measure. Let us denote by $a_k$ the equivalence class of $A_k$ in $\A$.  Also, for each $k\in\N$, let us define $L^2_k:=L^2(A_k)$ considering the trace of $\F$ on $A_k$ and the conditional probability $\PP(\cdot|A_k)$. For each $x\in L^2_k$ we denote $\Vert x|A_k\Vert_2:=\E_\PP[x^2|A_k]^{1/2}$. In this case, $L^0(\F)=\{ \sum_{k\in\N}1_{a_k}r_k\:;\: r_k\in\R\}$, and inspection shows $L^2_\F(\EE)=\{\sum_{k\in\N}1_{a_k}x_k\:;\: x_k|_{A_k}\in L^2_k\}$ and, for $x\in L^2_\F(\EE)$, we have $\Vert x|\F\Vert_2=\sum_{k\in\N}1_{a_k}\Vert x|A_k\Vert_2$.

For each $k$, let us choose a countable set $\{y^k_n\:;\:n\in\N\}$ with $y^k_n|_{A_k}=z^k_n$ where $\{z^k_n\}$ is a linearly  independent subset of $L^2_k$. Let
\[
U:=\sum_{k\in\N} 1_{a_k} U_{\{y^k_1,...,y^k_k\},1}.
\]
We have that $U$ is a neighborhood of $0\in L^2_\F(\EE)$ for the topology $\sigma(L^2_\F(\EE),L^2_\F(\EE))_{cc}$, but not for the topology $\sigma(L^2_\F(\EE),L^2_\F(\EE))$. Indeed, to reach a contradiction, let us assume that there is a finite subset $F$ of $L^2_\mathcal{F}(\mathcal{E})$ and $\varepsilon\in L^0_{++}(\mathcal{F})$, where $\varepsilon=\sum_{k\in\N} 1_{a_k}r_k$ with $r_k\in\R^+$, so that $U_{F,\varepsilon}\subset U$. Now, let us take $k:=\#F + 1$. 


Let us define $F_k:=\{y|_{A_k}\:;\: y\in F\}$, then
\[
\begin{array}{cc}
U_{F_k,r_k}\subset U_{\{z^k_1,...,z^k_k\},1} & \textnormal{ in }L^2_k.
\end{array}
\]
For each $y\in L^2_k$, let us denote $\mu_y(x):=\E_{\PP}[x y|A_k]$. Then, it follows that
\[
\bigcap_{y\in F} \ker (\mu_y) \subset \bigcap_{i=1}^k \ker (\mu_{y^k_i}).
\]
But this is impossible, because $\bigcap_{y\in F} \ker (\mu_y)$ is a vector subspace of $L^2_k$ with codimension less than $k$, included into a vector subspace with codimension $k$. 
\end{example}

We have the following result, which relates the stability on the topological structure of $E[\mathscr{T}]$ with the stability on the algebraic structure of $E[\mathscr{T}]^*$:

\begin{prop}
\label{prop: conj}
Let $E[\mathscr{T}]$ be a topological $L^0$-module. If $\mathscr{T}$ is stable, then $E[\mathscr{T}]^*$ is stable.   
\end{prop}
\begin{proof}
 Suppose that $\{\mu_k\}$ is a countable family of continuous $L^0$-linear applications from $E$ to $L^0$ and $\{a_k\}\in p(1)$, then we can define $\mu:=\sum 1_{a_k}\mu_k$, which is a $L^0$-linear application from $E$ to $L^0$. Let us show that $\mu$ is continuous. It suffices to study the continuity at $0\in E$. Fixed $\varepsilon\in L^0_{++}$, for each $k\in\N$ there exists $O_k\in\mathscr{T}$ with $0\in O_k$ so that $\mu(O_k)\subset B_\varepsilon$. If we set $O:=\sum 1_{a_k}O_k$, it is an open neighborhood of $0\in E$ since $\mathscr{T}$ is stable. We obtain $\mu(O)\subset \sum 1_{a_k}\mu(O_k) \subset B_\varepsilon$.
\end{proof}


\subsection{Relation between the class of locally $L^0$-convex modules and the class of conditionally locally convex vector spaces}

Once our study on locally $L^0$-convex modules has finished, we turn to recall the basic notions of the theory of conditional sets, which will be the setting used in the remainder of this paper. We will end this section by showing how the notion of locally $L^0$-convex module (with the suitable stability properties) is embedded in this setting.  
Let us recall the notion of conditional set: 

\begin{defn}
\cite[Definition]{key-7}
\label{defn: condSet} A conditional set $\textbf{E}$ of a non-empty set $E$ is a collection $\textbf{E}$ of objects $x|a$ for $x\in E$ and $a\in\A$ satisfying the following three axioms: 
  
\begin{enumerate}
	\item If $x|a=y|b$, then  $a=b$; 
	\item (Consistency) if $x,y\in E$ and $a,b\in\A$ with $a\leq b$, then $x|b=y|b$ implies $x|a=y|a$;
	\item (Stability) if $\{a_k\}\in p(1)$ and $\{x_k\}\subset E$, then there exists exactly one element $x\in E$ such that $x|a_k=x_k|a_k$ for all $k\in\N$. 
\end{enumerate}  
The unique element $x$ provided by 3 is called the \textit{concatenation of} $\{x_k\}$ along $\{a_k\}$, and is denoted by $\sum_{k\in\N} x_k|a_k$, or simply $\sum x_k|a_k$.
\end{defn}

Drapeau et al.\cite{key-7} originally introduced the notion of conditional set on an arbitrary complete Boolean algebra. However, for the purpose of this paper it suffices to consider the measure algebra $\A$ as the underlying Boolean algebra. Notice that, in doing so, we avoid to use non countable partitions, overcoming the difficulties arisen from that.

Another remark is that Drapeau et al.\cite{key-7} provided a general construction for the \textit{conditional real numbers} on an arbitrary Boolean algebra (see \cite[Definition 4.3]{key-7}).  In \cite[Theorem 4.4]{key-7}, it was also showed that, in the particular case in which the underlying Boolean algebra is a measure algebra $\A$,  conditional real numbers can be identified by the following conditional set:  

 On $L^0\times\A$ it can be defined an equivalence relation $\sim$ given by $(\eta, a)\sim(\xi, b)$ if, and only if, $1_a \eta = 1_b \xi$ and $a = b$. 
If we denote by $\eta|a$ the equivalence class of $(\eta,a)$, we have that the related quotient set $\textbf{R}$ is a conditional set of $L^0$ which is referred to as \textit{conditional real numbers}.

We can extend the equivalence relation defined above to $\bar{L^0}\times\A$, by identifying $(\eta, a),(\xi, b)$ whenever $1_a \eta = 1_b \xi$ and $a = b$, understanding $0\cdot(+\infty)=0$. Then, inspection shows that the quotient set $\overline{\textbf{R}}$ is a conditional set which will be called \textit{extended conditional real numbers}.  

More notions from conditional set theory are the following:

Let $\textbf{E}$ be a conditional set of a non-empty set $E$. A non-empty subset $F$ of $E$ is called \textit{stable} if
\[
F=\left\{ \sum x_i|a_i \:;\: \{a_i\}\in p(1),\: x_i\in F\textnormal{ for all }i \right\}.
\]
$S(\textbf{E})$ stands for the set of all stable subsets $F$  of $E$.

The \textit{stable hull} of a non-empty subset $F$ of $E$ is introduced in \cite{key-7} as 
\[
s(F):=\left\{\sum x_i|a_i \:;\: \{a_i\}\in p(1),\: x_i\in F\textnormal{ for all }i\right\},
\]
which is the smallest stable subset containing $F$.

It is known from \cite{key-7} that every set $F\in S(\textbf{E})$ generates a conditional set 
\[
\textbf{F}:=\left\{x|a\:;\: x\in F,\:a\in\A\right\}.
\]

For any non-empty subset $F$ of $E$, $\textbf{s}(F)$ stands for the conditional subset generated by $s(F)$.

\begin{example}
The conditional numbers $\textbf{R}$ is a conditional set of $L^0(\F)$. Let $L^0(\F;\N)$ and $L^0(\F;\Q)$, or simply $L^0(\N)$ and $L^0(\Q)$, denote the sets of (equivalence classes of) $\F$-measurable natural-valued random variables and rational-valued random variables, respectively. Then, if we consider $\N,\Q$ as subset of $L^0$ (by considering constant function in $\N$ or $\Q$, respectively),  it is clear that $s(\N)=L^0(\N)$ and $s(\Q)=L^0(\Q)$. Thereby, we define the conditional sets $\textbf{N}:=\textbf{s}(\N)$ and $\textbf{Q}:=\textbf{s}(\Q)$, which are called  \textit{conditional natural numbers} and \textit{conditional rational numbers}, respectively.

Conditional natural numbers and conditional rational numbers were defined in \cite[Examples 2.3]{key-7} in a general way, but it is also clear from \cite{key-7} that they can be identified with the above conditional sets, and this is the definition that we adopt for the present setting. 
\end{example}

For a given conditional set $\textbf{E}$, we have that $P(\textbf{E})$ denotes the collection of all conditional sets $\textbf{F}$ generated by $F\in S(\textbf{E})$, and the \textit{conditional power set} is defined by   
\[
\textbf{P}(\textbf{E}):=\left\{\textbf{F}|a=\{x|b\:;\: x\in F,\:b\leq a\}\:;\:\textbf{F}\in P(\textbf{E}),\:a\in\A\right\},
\] 
which is a conditional set of $P(\textbf{E})$ (see \cite[Definition 2.7]{key-7}).

Drapeau et al.\cite{key-7} also observed that every element $\textbf{F}|a$ is a conditional set of $F|a:=\{x|a\:;\: x\in F\}$ considering the measure algebra $\A_a$,  with the conditioning $(x|a)|b:=x|b$ for $b\leq a$. Such conditional sets are called \textit{conditional subsets} of $\textbf{E}$. Also, we say that $\textbf{F}|a$ is on $a$. 

Suppose that $\textbf{L}|a$ and $\textbf{M}|b$ are conditional subsets. Then, $\textbf{L}|a$ is said to be conditionally included, or conditionally contained, in $\textbf{M}|b$ if $\textbf{L}|a\subset\textbf{M}|b$. In that case, we use the notation $\textbf{L}|a\sqsubset\textbf{M}|b$, which defines a partial order in $\textbf{P}(\textbf{E})$.

As in \cite{key-7}, the notation $\textbf{F}$ or $\textbf{F}|a$ for a conditional subset on $a$ will be chosen depending on the context.

Drapeau et al.\cite{key-7} introduced operations for conditional subsets: we have the \textit{conditional union}, the \textit{conditional intersection}, and the \textit{conditional complement}, which are denoted by $\sqcup$, $\sqcap$ and ${}^\sqsubset$, respectively. We do not include the construction; instead, we refer to the proof of \cite[Theorem 2.9]{key-7}. Moreover, $(\textbf{P}(\textbf{E}), \sqcup,\sqcap,{}^\sqsubset, \textbf{E}, \textbf{E}|0)$ is a complete Boolean algebra which is called the \textit{algebra of conditional sets} (see \cite[Corollary 2.10]{key-7}).

Let $\textbf{E}$ and $\textbf{F}$ be conditional sets of $E$ and $F$, respectively. A function $f:E\rightarrow F$ is \textit{stable} if $f(\sum x_k|a_k)=\sum f(x_k)|a_k$ for every $\{x_k\}\subset E$ and $\{a_k\}\in p(1)$. A conditional function $\textbf{f}:\textbf{E}\rightarrow\textbf{F}$ is the conditional graph $\textbf{G}_\textbf{f}=\left\{ (f(x)|a,x|a) \:;\:  x\in E,\: a\in\A\right\}$ defined from  on a stable function $f:E\rightarrow F$. Let $\textbf{E}$ and $\textbf{I}$ be conditional sets of $E$ and $I$, respectively. A stable family $\{x_i\}$ in $E$ is the graph $G_f=\left\{(f(i),i) \:;\: i\in I\right\}$ where $f:I\rightarrow E$ is a stable function. A conditional family $\{\textbf{x}_\textbf{i}\}$ in $\textbf{E}$ is a conditional function $\textbf{f}:\textbf{I}\rightarrow\textbf{E}$ (see Definitions 2.17 and 2.20 of \cite{key-7}). 

In general, conditional set theory is an extensive theoretical development, therefore there is no room in this paper to give a detailed explanation.  Instead, we refer the reader to \cite{key-7}, where a comprehensive introduction to conditional set theory is provided.  There, the algebra of conditional set is introduced, and apart from the notions  defined above, one can find all sort of conditional versions of classical concepts such as: \textit{conditional binary relation}, \textit{conditional partial order}, \textit{conditional supremum} (resp. \textit{infimum}), \textit{conditional total order}, \textit{conditional image} and  \textit{conditional antiimage} of a conditional function, \textit{conditionally injective function}, \textit{conditionally surjective function}, \textit{conditionally bijective function} (see \cite[Section 2]{key-7}); \textit{conditional topology}, \textit{conditional topological space}, \textit{conditional topological base}, \textit{conditional closure}, \textit{conditional interior}, \textit{conditional neighborhood base}, \textit{conditionally continuous function}, \textit{conditional convergence}, \textit{conditional compactness}, \textit{conditionally sequentially compactness} (see \cite[Section 3]{key-7}); in \cite[Section 4]{key-7}  \textit{conditional metric spaces} are introduced. 

Also, in \cite[Section 5]{key-7} it was introduced the notion \textit{conditionally vector space} (see \cite[Definition 5.1]{key-7}), and, as particular case, it was also defined \textit{conditionally locally convex vector spaces} (see \cite[Definition 5.4]{key-7}), which will be called \textit{conditionally locally convex spaces} hereafter. Other important notions introduced in \cite[Section 5]{key-7} are the notions of \textit{conditionally convex}, \textit{conditionally absorbent} and \textit{conditionally balanced subset} (see \cite[Definition 5.1]{key-7}). The conditional weak topologies were also introduced in \cite[Section 5]{key-7}. In addition, Zapata \cite{key-38} provided a study of conditional weak topologies based on conditionally locally convex spaces.

In what follows, we adopt the notation and terminology of \cite{key-7} with only three exceptions that will be explained below (which is exactly the notation of \cite{key-38}):

\begin{enumerate}
	\item For given finitely many conditional sets $\textbf{E}_1,\textbf{E}_2,...,\textbf{E}_n$ the conditional product will be denoted by $\textbf{E}_1\Join...\Join\textbf{E}_n$. This notation allows us to distinguish between classical product and conditional product.
	\item Drapeau et al.\cite{key-7} introduced the notion of \textit{conditional element} of a conditional set $\textbf{E}$ of $E$. Namely, a conditional element is the object $\textbf{x}:=\left\{x|a\:;\: a\in\A\right\}$ where $x\in E$. A conditional element is not an element of $\textbf{E}$. However, by consistency, the map that sends the conditional element $\textbf{x}$ to the element $x|1$, is a bijection. Therefore we can make the identification $\textbf{x}\equiv x|1$. Doing so, we can see $\textbf{x}$ as an element of $\textbf{E}$ and use the convenient notation $\textbf{x}\in\textbf{E}$.
	\item Some conditional subsets will be required to be defined by describing their conditional elements. For instance, suppose that $\phi$  is a certain statement which can be true of false for the conditional elements of $\textbf{E}$. Since the family $\{\textbf{x} \in \textbf{E} \:;\: \phi(\textbf{x})\textnormal{ is true}\}$ is not generally a conditional set, we will use the following formal set-builder notation for conditional subsets:
\[
                    [\textbf{x} \in \textbf{E} \:;\: \phi(\textbf{x})\textnormal{ is true}] := \textbf{s} \left(\left\{x \:;\: \phi(\textbf{x})\textnormal{ is true} \right\}\right),
\]  
provided that $\phi$ is true for at least a conditional element $\textbf{x}\in\textbf{E}$ (recall that the stable hull is defined for non-empty sets).
\end{enumerate}

By means of the formal set-builder notation for conditional subsets we define the following conditional subsets of $\textbf{R}$:
  
$\textbf{R}^+:=\left[\textbf{r}\in\textbf{R}\:;\:\textbf{r}\geq \textbf{0}\right]$, $\textbf{R}^{++}:=\left[\textbf{r}\in\textbf{R}\:;\:\textbf{r}>\textbf{0}\right]$, $\overline{\textbf{R}}^+:=\left[\textbf{r}\in\overline{\textbf{R}}\:;\: \textbf{r}\geq \textbf{0}\right]$ and  $\overline{\textbf{R}}^{++}:=\left[\textbf{r}\in\overline{\textbf{R}}\:;\: \textbf{r}>\textbf{0}\right]$.

Given a conditional real number $\textbf{r}=r|1$ we define the conditional inverse as the following conditional real number $\textbf{r}^{-1}:=1_{(r\neq 0)}(r + 1_{(r=0)})^{-1}|1$.

A conditional set $\textbf{E}$ is said to be \textit{conditionally countable} if there is a conditionally injective function $\textbf{f}:\textbf{E}\rightarrow \textbf{N}$; we say that $\textbf{E}$ is \textit{conditionally finite} if there exists a conditional bijection $\textbf{f}:\textbf{E}\rightarrow [\textbf{k}\:;\: \textbf{1}\leq\textbf{k}\leq\textbf{n}]$ for some $\textbf{n}\in\textbf{N}$ (see \cite[Definition 2.23]{key-7}).

By definition, we know that if $(\textbf{E},\boldsymbol{+},\boldsymbol{\cdot})$ with $\boldsymbol{+}:\textbf{E}\Join\textbf{E}\rightarrow\textbf{E}$ and $\boldsymbol{\cdot}:\textbf{R}\Join\textbf{E}\rightarrow\textbf{E}$ is a conditional vector space, then $(E,+,\cdot)$ with $+:E\times E\rightarrow E$ and $\cdot:L^0\times E\rightarrow E$ is a $L^0$-module.  

We have the following result which is almost obvious, but will be useful later:

\begin{prop}
\label{lem:connex}
Let $\textbf{E}$ be a conditional vector space. Let us define the following equivalence relation on $E\times\A$: $(x,a)\sim (y,b)$ if, and only if, $a=b$ and $1_a x=1_b y$. Then, the quotient set $\textbf{F}:=E\times\A/\sim$ is a conditional set of $E$ such that the identity $id:E\rightarrow E$ is a stable function when considering $E$ in the left side as the generating set of $\textbf{E}$ and in the right side as the generating set of $\textbf{F}$. Further, $\textbf{F}$ is in fact a conditional vector space and the corresponding conditional function $\textbf{id}:\textbf{E}\rightarrow\textbf{F}$ is a conditional ismorphism of conditional vector spaces.

In addition, we have that $E$ is a stable $L^0$-module, and for every sequence $\{x_k\}\subset E$ and partition $\{a_k\}\in p(1)$ we have $\sum x_k|a_k=\sum 1_{a_k}x_k$ (for both structures of conditional sets $\textbf{E}$ and $\textbf{F}$).
\end{prop}
\begin{proof}

The conclusions of the statement follows easily by inspection by taking into account the following equivalence

\begin{equation}
\label{x0}
x|a=y|b\textnormal{ if, and only if, }a=b\textnormal{ and }1_a x= 1_b y.
\end{equation}

So, let us show (\ref{x0}). Indeed, we know that for any $a\in\A$, we have $1_a=1|a + 0|a^c$ in $L^0$. Then, by using that the scalar product $\cdot:L^0\times E\rightarrow E$ is stable, we obtain:
 
\begin{equation}
\label{x1}
1_a x=(1|a + 0|a^c)\cdot(x|a + x|a^c)=1\cdot x|a + 0\cdot x|a^c=x|a + 0|a^c.
\end{equation}

Similarly,

\begin{equation}
\label{x2}
1_b y= y|b+ 0|b^c.
\end{equation}

Let us suppose $x|a=y|b$. By 1 of Definition \ref{defn: condSet} we have $a=b$. Then, in view of (\ref{x1}) and (\ref{x2}), we obtain $1_a x= 1_b y$. 

Conversely, suppose $1_a x=1_b y$ and $a=b$. Again, by (\ref{x1}) and (\ref{x2}), we have that $x|a + 0|a^c=1_a x=1_a y=y|a + 0|a^c$. Since $a=b$, it follows $x|a=y|b$.
\end{proof}
 
In \cite[Definition 3.14]{key-38} it was introduced the notion of \textit{conditional seminorm}. In this setting where $\A$ is the measure algebra, we find that a seminorm $\nVert\cdot\nVert:\textbf{E}\rightarrow\textbf{R}^+$ defined on a conditional vector space $\textbf{E}$ is a conditional function such that the generating stable function $\Vert\cdot\Vert:E\rightarrow L^0_+$ is in fact a $L^0$-seminorm. 

Based on conditional set theory, for a given conditional vector space $\textbf{E}$ and a conditional family of conditional seminorms $\mathcal{P}$, we can endow $\textbf{E}$ with a conditional topology as follows:

Given a conditionally finite conditional subset $\textbf{F}\sqsubset \mathcal{P}$ on $1$. Let us fix $\textbf{r}\in\textbf{R}^{++}$. We define:
\[
\textbf{U}_{\textbf{F},\textbf{r}}:=\left[\textbf{x}\in\textbf{E}\:;\: \nVert\textbf{x}\nVert\leq \textbf{r},\:\forall \nVert\cdot \nVert\in\textbf{F} \right].
\]
We know from \cite{key-38} that the conditional set 
\[
\mathcal{U}:=\left[\textbf{x}+\textbf{U}_{\textbf{F},\textbf{r}}\:;\: \textbf{r}\in\textbf{R}^{++},\: \textbf{F}\sqsubset\mathcal{P}\textnormal{ conditionally finite} \right]
\]
is a conditional topological neighborhood base for a conditionally locally convex topology $\mathcal{T}$. 

 For a  conditionally locally convex space $\textbf{E}[\mathcal{T}]$, let $\textbf{E}[\mathcal{T}]^*$ denote —or simply $\textbf{E}^*$—, the conditional vector space of conditionally linear and conditionally continuous functions $\textbf{f}:\textbf{E}\rightarrow\textbf{R}$. We define the conditional weak topology $\sigma(\textbf{E},\textbf{E}^*)$ on $\textbf{E}$ as the conditionally locally convex topology induced by the conditional family of conditional seminorms $\left[\textbf{p}_{\textbf{x}^*}\:;\:\textbf{x}^*\in \textbf{E}^* \right]$ defined by $\textbf{p}_{\textbf{x}^*}=|\textbf{x}^*(\textbf{x})|$ for $\textbf{x}\in\textbf{E}$. Analogously, the conditional weak-$*$ topology $\sigma(\textbf{E}^*,\textbf{E})$ on $\textbf{E}^*$ is defined.

We have below the main theorem of this section, which states in terms of equivalence of categories the connection between locally $L^0$-convex modules and conditionally locally convex spaces. The proof of this result rely on \cite[Proposition 3.5]{key-7}, where a connection between conditional topologies and classical topologies is provided.

\begin{thm}
\label{thm: conection}
Let $\A$ be the measure algebra associated to some probability space $(\Omega,\F,\PP)$; let $\mathcal{C}$ be the category whose objects are conditionally locally convex spaces with underlying measure algebra $\A$, and whose morphisms are conditionally continuous  linear applications; and let $\mathscr{C}$ be the category whose objects are locally stable $L^0(\F)$-convex modules, and whose morphisms are continuous $L^0(\F)$-linear functions.

Let $\phi:\mathcal{C}\rightarrow \mathscr{C}$ be the functor defined by 
\begin{enumerate}
	\item $\phi(\textbf{E}[\mathcal{T}])=E[\mathscr{T}]$, where $\textbf{E}[\mathcal{T}]$ is a conditionally locally convex space, $E$ is the generating set of $\textbf{E}$ and    
\[
\mathscr{T}:=\left\{O\in S(\textbf{E}) \:;\: \textbf{O}\in\mathcal{T}\right\};
\]
	\item and $\phi(\textbf{f})=f$ for every conditionally continuous  linear function $\textbf{f}:\textbf{E}\rightarrow\textbf{F}$, where $f:E\rightarrow F$ is the generating stable function of $\textbf{f}$. 
\end{enumerate}
Then, $\phi$ is well defined and is an equivalence of categories between $\mathcal{C}$ and $\mathscr{C}$.

Moreover, if $\mathscr{P}$ is a family of $L^0(\F)$-seminorms stably inducing $\mathscr{T}$, then the conditional family of conditional seminorms $\mathcal{P}:=\left[\nVert\cdot\nVert \:;\:\Vert\cdot\Vert\in \mathscr{P}\right]$ induces $\mathcal{T}$.

In addition, we have the following relation
\begin{equation}
\label{eq: weak}
\begin{array}{ccc}
\phi(\textbf{E}[\sigma(\textbf{E},\textbf{E}^*)])=E[\sigma(E,E^*)_{cc}] & \textnormal{ and } & \phi(\textbf{E}^*[\sigma(\textbf{E}^*,\textbf{E})])=E^*[\sigma(E^*,E)_{cc}]. 
\end{array}
\end{equation}

\end{thm}
\begin{proof}
First, let us show that $\phi$ is well defined. Indeed, if $\textbf{E}[\mathcal{T}]$ is a conditionally locally convex space, by  Proposition \ref{lem:connex}, we know that $E$ is a stable $L^0$-module. Besides, we know from \cite{key-38} that a conditionally locally convex topology is always induced by a family of conditional seminorms. Therefore, there exists a conditional topological neighborhood base $\mathcal{U}$ of $\textbf{0}\in\textbf{E}$ such that every $\textbf{U}\in\mathcal{U}$ is conditionally convex, conditionally absorbent and conditionally balanced. In view of \cite[Proposition 3.5]{key-7}, the stable collection $\mathscr{U}:=\{U\in S(\textbf{E})\:;\: \textbf{U}\in\mathcal{U}\}$ is a neighborhood base of the topology $\mathscr{T}=\left\{O\in S(\textbf{E}) \:;\: \textbf{O}\in\mathcal{T}\right\}$. Further, every $U\in\mathscr{U}$ is stable, $L^0$-convex, $L^0$-absorbent and $L^0$-balanced, because $\textbf{U}$ is conditionally convex, conditionally absorbent, and conditionally balanced. We conclude that $E[\mathscr{T}]$ is a locally stable $L^0$-convex module.

Further, for a given  conditionally linear continuous function $\textbf{f}:\textbf{E}\rightarrow\textbf{F}$ between conditionally locally convex spaces, the stable function $f: E\rightarrow F$ is clearly $L^0$-linear. Let us show that $f$ is continuous. It suffices to show that $f$ is continuous at $0\in E$. Indeed, let $V\in S(\textbf{F})$ be an open neighborhood of $0\in F$. Since $\textbf{f}$ is conditionally continuous there exists a conditionally open set $\textbf{O}\sqsubset \textbf{E}$ on $1$ with $\textbf{0}\in\textbf{O}$ such that the conditional image $\textbf{f}(\textbf{O})\sqsubset \textbf{V}$ and, consequently, $f(O)\subset V$.  

Now, let us show that $\phi$ is an equivalence of categories. Indeed, let us consider the following functor $\psi:\mathscr{C}\rightarrow\mathcal{C}$, which we describe as follows: First, given a locally stable $L^0$-convex module $E[\mathscr{T}]$, as in Proposition \ref{lem:connex}, we can define an equivalence relation on $E\times\A$ where the equivalence class of $(x,a)$ is given by $x|a:=\left\{(y,b)\in E\times\A\:;\:a=b,1_a x=1_b y\right\}$. Due to the stability of $E$, we obtain that the quotient $\textbf{E}$ is a conditional vector space. Let us put $\psi(E[\mathscr{T}]):=\textbf{E}[\mathcal{T}]$, with $\mathcal{T}:=[\textbf{O}\:;\:O\in\mathscr{T}]$. Now, let $\mathscr{U}$ be a neighborhood base of $0\in E$ as in Theorem \ref{thm: caracterizacionII}, which is a stable collection of stable subsets of $E$. Then, \cite[Proposition 3.5]{key-7} yields  that $\mathcal{T}$ is a conditional topology and $\mathcal{U}=\left[\textbf{U}\:;\:U\in \mathscr{U}\right]$ is a conditional neighborhood base of $\textbf{0}\in\textbf{E}$. Since every $U\in \mathscr{U}$ is $L^0$-convex, we have that each $\textbf{U}\in\mathcal{U}$ is conditionally convex. We conclude that $\textbf{E}[\mathcal{T}]$ is a conditionally locally convex space. Second, for each continuous $L^0$-linear function $f:E\rightarrow F$ between locally stable $L^0$-convex modules, we know that $f$ is stable, and  it therefore generates a conditional function $\textbf{f}:\textbf{E}\rightarrow\textbf{F}$ between conditionally locally convex vector spaces. Thus, we define $\psi(f):=\textbf{f}$. It is easy to show that $\textbf{f}$ is conditionally linear and conditionally continuous.

Let us show that the product $\phi\psi$ is naturally equivalent to the identity functor $Id_{\mathscr{C}}$, and  $\psi\phi$ to the identity functor $Id_{\mathcal{C}}$. Indeed, first it is clear that $\phi\psi=Id_{\mathscr{C}}$. Second, regarding $\psi\phi$, it is easy to show that the conditional isomorphism provided by Proposition \ref{lem:connex} defines a natural isomorphism between the functors $\psi\phi$ and $Id_{\mathcal{C}}$. Therefore, we conclude that $\mathscr{C}$ and $\mathcal{C}$ are equivalent categories.

Now, let us turn to show that if $\mathscr{P}$ is a family of $L^0$-seminorms stably inducing $\mathscr{T}$, then the conditional family of conditional seminorms $\mathcal{P}:=\left[\nVert\cdot\nVert \:;\:\Vert\cdot\Vert\in \mathscr{P}\right]$ induces $\mathcal{T}$. Indeed, for given $\{a_k\}\in p(1)$  and a sequence $\{\Vert\cdot\Vert_k\}_{k\in\N}$ in $\mathscr{P}$, let us define 
\[
\begin{array}{cc}
\Vert x\Vert_{\{\Vert\cdot\Vert_k\},\{a_k\}}:=\sum 1_{a_k}\left\Vert x\right\Vert_k, & \textnormal{ for }x\in E. 
\end{array}
\]
Inspection shows that $\Vert \cdot\Vert_{\{\Vert\cdot\Vert_k\},\{a_k\}}:E\rightarrow L^0_+$ is a $L^0$-seminorm. The collection of $L^0$-seminorms  $\Vert\cdot\Vert_{\{\Vert\cdot\Vert_k\},\{a_k\}}$, with $\{a_k\}\in p(1)$ and $\{\Vert\cdot\Vert_k\}\subset\mathscr{P}$,  is a stable family of $L^0$-seminorms, which defines a conditional family of conditional seminorms $\mathcal{P}$. Then, we claim that
\begin{equation}
\label{eq: seminorms}
\begin{array}{c}
\left\{ U_{\{F_k\},\{a_k\},r}  \:;\: r\in L_{++}^{0},\: \{F_k\}\subset \mathscr{P} \textnormal{ finite for all }k\in\N,\{a_k\}\in p(1) \right\}=\\
=\left\{U_{F,r}  \:;\: \textbf{F}\sqsubset\mathcal{P}\textnormal{ conditionally finite, }\textbf{r}>\textbf{0}\right\}.
\end{array}
\end{equation}
If we prove the above equality, from \cite[Proposition 3.5]{key-7},  it follows that $\mathcal{T}$ is the conditional topology induced by $\mathcal{P}$. Indeed, on one hand, let us take $\{a_k\}\in p(1)$,  a family $\{F_k\}$ of non-empty finite subsets of $\mathscr{P}$ and $r\in L_{++}^{0}$. For each $k\in\N$, let  $F_k=\{\Vert\cdot\Vert_1^k,...,\Vert\cdot\Vert_{n_k}^k\}$. Now, let us take $n:=\sum 1_{a_k} n_k\in L^0(\N)$ with $\{a_k\}\in p(1)$ and $\{n_k\}\subset\N$, which defines the conditional natural number $\textbf{n}=n|1\in\textbf{N}$. For each conditional natural number $\textbf{m}\leq \textbf{n}$ we have that $m=\sum_k 1_{a_k}\sum_{1\leq i\leq n_k} 1_{b_{i,k}}i\in L^0(\N)$ with $\{b_{i,k}\}_i\in p(a_k)$ for each $k\in\N$. We define $\Vert\cdot\Vert_m:=\Vert\cdot\Vert_{\{\Vert\cdot\Vert_i^k\},\{a_k\wedge b_{i,k}\}}$, which defines a conditional seminorm $\nVert\cdot\nVert_\textbf{m}\in\mathcal{P}$. Now, let us consider the conditionally finite subset $\textbf{F}:=\left[\nVert\cdot\nVert_\textbf{m}\:;\: \textbf{m}\leq \textbf{n}\right]$. Then it can be checked that ${U}_{F,r}=U_{\{F_k\},\{a_k\},r}$. 

On the other hand, let us take  $\textbf{F}\sqsubset \mathcal{P}$ on $1$, which is conditionally finite, and $\textbf{r}\in\textbf{R}^{++}$.  Since $\textbf{F}$ is conditionally finite, it is of the form $\textbf{F}=\left[\nVert\cdot\nVert_\textbf{m}\:;\:   \textbf{m}\leq \textbf{n}\right]$ for some $n=\sum 1_{a_k}n_k$ with $n_k\in\N$ and $\nVert\cdot\nVert_\textbf{m}$ in ${\mathcal{P}}$. For each $k\in\N$, let us define the finite set $F_k:=\{\Vert\cdot\Vert_m\:;\: m\in\N, m\leq n_k \}$. Then, inspection shows that $U_{F,r}=U_{\{F_k\},\{a_k\},r}$. Thereby, we obtain (\ref{eq: seminorms}).

%

%

Finally, if in particular we consider the families of $L^0$-seminorms $\sigma(E,E^*)$ and $\sigma(E^*,E)$, we obtain (\ref{eq: weak}) from the following equalities
 
\[
\begin{array}{ccc}
\sigma(\textbf{E},\textbf{E}^*)=[\nVert\cdot\nVert \:;\: \Vert\cdot\Vert\in \sigma(E,E^*)] & \textnormal{ and } &
\sigma(\textbf{E}^*,\textbf{E})=[\nVert\cdot\nVert \:;\: \Vert\cdot\Vert\in \sigma(E^*,E)].
\end{array}
\]
\end{proof}

Now, Theorem \ref{thm: conection} allows us to deal with locally $L^0$-convex modules by using the machinery of conditional set theory, and, conversely, to draw from $L^0$-theory some existing important results  (for example \cite{key-10,key-5,key-6,key-13}) to the conditional setting when the underlying boolean algebra is a measure algebra. For instance, we have a Fenchel-Morau theorem (see \cite[Theorem 3.8]{key-10}) which can be transcribed by using the language of conditional sets. Also, Theorem \ref{thm: conection} provides a method to indentify the conditional topological dual $\textbf{E}^*$ of a conditionally locally convex space $\textbf{E}$ when one know the topological dual $E^*$ of the $L^0$-module $E$.



%
%

We have the following example:

\begin{example}
\label{exam: Lp}

For $1\leq p\leq\infty$, let us consider the locally $L^0$-convex module $L^p_\F(\EE)$. We can define the conditional normed space $\textbf{L}^p_\F(\EE):=\psi(L^p_\F(\EE))$, where $x|a:=\{(y,b); 1_a x = 1_b y \textnormal{ and }a=b\}$. Also, the $L^0$-norm $\Vert\cdot\Vert_p$ defines a conditional norm $\nVert\cdot\nVert_p:\textbf{L}^p_\F(\EE)\rightarrow\textbf{R}^+$, which induces a conditional topology.

For $1\leq p<+\infty$, if $1<q\leq+\infty$ with $1/p+1/q=1$, let us consider the $L^0$-isometric isomorphism $T:L^p_\F(\EE)\rightarrow L^q_\F(\EE)$, $z\mapsto T_z$.  Then, we can consider the conditional isometric isomorphism $\psi(T)=\textbf{T}$ between conditional normed spaces, which, in view of Theorem \ref{thm: conection}, are precisely $(\textbf{L}^p_\F(\EE))^*$ and $\textbf{L}^q_\F(\EE)$.     
\end{example}

Finally, we would like to highlight that Examples  \ref{ex3}, \ref{ex4}, \cite[Example 2.12]{key-10}, \cite[Example 1.1]{key-15}, and  \cite[Example 2.7]{key-17} provide $L^0$-modules which, lacking stability in either the algebraic or the topological structure, fall outside the scope of conditional set theory in view of Theorem \ref{thm: conection}. Therefore it is open up to find an analytic approach to these 'sick' $L^0$-modules, which might eventually be considered as model space for the financial applications.  

Let us see that there is some hope for these boundary examples. Given any locally $L^0$-convex module $E[\mathscr{T}]$, we can construct from it a locally stable $L^0$-convex module $\bar{E}[\bar{\mathscr{T}}]$, for which Theorem \ref{thm: conection} applies. Indeed, in a  first step, in order to guarantee the uniqueness of concatenations, let us define the following equivalence relation $\sim$ on $E$:
\[
\begin{array}{ccc}
x\sim y & 	\textnormal{whenever} & 1_{a_k}x=1_{a_k}y \textnormal{ for some }\{a_k\}\in p(1).
\end{array}
\]
Now, let us consider the quotient set $\tilde{E}$. Let us denote by $\overline{x}$ the class of $x$. It easy to check that the operations $\overline{x}+\overline{y}:=\overline{x+y}$ and  $\lambda\overline{x}=\overline{\lambda x}$ are well defined and $\tilde{E}$, endowed with this operations, is a $L^0$-module.

In a second step, in order to stabilize the algebraic structure, let us consider the set
\[
\left\{(x_k,a_k)\:;\: \{a_k\}\in p(1),\:\{x_k\}\subset \tilde{E} \right\},
\]
and the equivalence relation 
\[
\begin{array}{ccc}
(x_i,a_i)\sim(y_j,b_j) & \textnormal{whenever} & 1_{a_i\wedge b_j}x_i=1_{a_i\wedge b_j}y_j \textnormal{ for all }i,j\in\N.
\end{array}
\]

We consider the quotient set $\bar{E}$. Let us denote by $[x_k,a_k]$ the class of $(x_k,a_k)$. We define the operations
\[
\begin{array}{ccc}
[x_i,a_i]+[y_j,b_j]=[x_i+y_j,a_i\wedge b_j], & \textnormal{and} & \lambda[x_i,a_i]=[\lambda x_i,a_i]. 
\end{array}
\]

They are well defined and inspection shows that $\bar{E}$ is a stable $L^0$-module, where 
\[
\sum_{k\in\N} 1_{a_k}[x^k_n,b^k_n]=[x^k_n,a_k\wedge b_n^k].
\]

The third step is to stabilize the topology $\mathscr{T}$. For that, let us consider the application
\[
\begin{array}{cc}
\pi:E\rightarrow \bar{E}, & x\mapsto [\overline{x},1].
\end{array}
\]

Let us take some neighborhood base $\mathscr{U}$ of $0\in E$, such that each $U\in\mathscr{U}$ is $L^0$-convex, $L^0$-absorbent and $L^0$-balanced. Then, let
\[
\bar{\mathscr{U}}:=\left\{ \sum_{k\in\N}1_{a_k}\pi(U_k)\:;\: U_k\in\mathscr{U}\textnormal{ for all }k,\:\{a_k\}\in p(1) \right\}.
\] 

We have that $\bar{\mathscr{U}}$ is a topological basis of some topology $\bar{\mathscr{T}}$. Moreover, $\bar{E}[\bar{\mathscr{T}}]$ is a  locally stable $L^0$-convex, which will be referred to as the \textit{stable closure} of $E[\mathscr{T}]$, and $\pi:E[\mathscr{T}]\rightarrow\bar{E}[\bar{\mathscr{T}}]$ is $L^0$-linear and open; injective whenever there is uniqueness in concatenations; and bijective whenever $E$ is stable.


\section{Conditional version of James Theorem}   


The aim of this section is to prove a conditional  James' compactness theorem in the non linear setting discussed in \cite{key-22}. In pursuing this goal, we first need some preliminary results.  

Let us first recall the notion of conditional sequence, which was introduced in \cite{key-7}. Namely, a \textit{conditional sequence} in a conditional set $\textbf{E}$ is a conditional family $\{\textbf{x}_\textbf{n}\}_{\textbf{n}\in\textbf{N}}$ of elements of $\textbf{E}$. If $\{\textbf{n}_\textbf{k}\}_{\textbf{k}\in\textbf{N}}$ is a conditional sequence in $\textbf{N}$ such that $\textbf{k}<\textbf{k}'$ implies that $\textbf{n}_\textbf{k}<\textbf{n}_\textbf{k}'$, then $\{\textbf{x}_{\textbf{n}_\textbf{k}}\}_{\textbf{k}\in\textbf{N}}$ is another conditional sequence which is said to be a \textit{conditional subsequence} of $\{\textbf{x}_\textbf{n}\}$. It is not difficult to verify that $\{\textbf{n}_\textbf{k}\}$ is conditionally cofinal in the sense that for any $\textbf{n}\in\textbf{\textbf{N}}$ there is another $\textbf{k}\in\textbf{N}$ such that $\textbf{n}_\textbf{k}\geq\textbf{n}$.

An important remark is that, for a given traditional sequence $\{x_n\}_{n\in\N}$ in $E$, we can construct a conditional sequence as follows: for any conditional natural number $\textbf{n}$ with $n:=\sum n_k|a_k$, $n_k\in\N$ and $\{a_k\}\in p(1)$, we can define a stable function from $L^0(\N)$ to $E$ given by $x_n:=\sum x_{n_k}|a_k$ for $n\in L^0(\N)$. Then, the associated conditional function defines a conditional sequence $\{\textbf{x}_\textbf{n}\}$ in $\textbf{E}$. Moreover, if $\{x_{n_k}\}_{k\in\N}$ is a subsequence of $\{x_n\}$ and we apply the method described above for constructing a conditional sequence $\{\textbf{x}_{\textbf{n}_\textbf{k}}\}_{\textbf{k}\in\textbf{N}}$, we obtain a conditional subsequence of $\{\textbf{x}_\textbf{n}\}$. 

Let $\{\textbf{x}_\textbf{n}\}$ be a conditional sequence  in $\textbf{R}$. We define $\underset{\textbf{n}}\nlimsup \textbf{x}_\textbf{n}=\underset{\textbf{m}}\ninf\:\underset{\textbf{n}\geq\textbf{m}}\nsup \textbf{x}_\textbf{n}$ and $\underset{\textbf{n}}\nliminf \textbf{x}_\textbf{n}=\underset{\textbf{m}}\nsup\:\underset{\textbf{n}\geq\textbf{m}}\ninf \textbf{x}_\textbf{n}$. So, it can be checked that there exists $\nlim_\textbf{n} \textbf{x}_\textbf{n}=\textbf{x}$ (with the conditional topology of $\textbf{R}$ introduced in (iii) of \cite[Definition 4.3]{key-7}) if, and only if, $\underset{\textbf{n}}\nlimsup \textbf{x}_\textbf{n}=\textbf{x}=\underset{\textbf{n}}\nliminf \textbf{x}_\textbf{n}$ and if, and only if, the sequence $\{x_n\}_{n\in\N}$ converges almost surely to $x$ in $L^0(\F)$. 

Let $\textbf{E}$ be a conditional vector space and two conditionally finite families $[\textbf{x}_\textbf{n}\:;\: \textbf{n}\leq \textbf{m}]\sqsubset\textbf{E}$ and $[\textbf{r}_\textbf{n}\:;\: \textbf{n}\leq \textbf{m}]\sqsubset\textbf{R}$. For some conditional natural number $\textbf{m}$ with $m=\sum_{k\in\N} m_k|a_k$ we denote by
\[
\sum_{1\leq\textbf{n}\leq \textbf{m}} \textbf{r}_\textbf{n}\textbf{x}_\textbf{n}
\]
 the conditional real number generated by $\sum_{k\in\N}\left(\sum_{1\leq n \leq m_k} r_k x_k\right)|a_k.$

Given a conditional sequence $\{\textbf{x}_\textbf{n}\}$, we have that the partial sums $\textbf{s}_\textbf{m} := \sum_{1\leq \textbf{n} \leq \textbf{m}}  \textbf{x}_\textbf{n}$ define a conditional sequence $\{\textbf{s}_\textbf{m}\}$. Then, we understand an infinite sum of the conditional sequence $\{\textbf{x}_\textbf{n}\}$ as the following conditional limit 
\[
\sum_{\textbf{n}\geq \textbf{1}} \textbf{x}_\textbf{n} := \nlim_{\textbf{m}} \textbf{s}_\textbf{m}. 
\]  

Given a conditional sequence $\{\textbf{f}_\textbf{n}\}$ of conditional functions $\textbf{f}_\textbf{n}:\textbf{E}\rightarrow\textbf{R}$ defined on a conditional set $\textbf{E}$, such that  for each  $\textbf{x}$ in $\textbf{E}$, it holds that $\underset{\textbf{n}\in\textbf{N}}\nsup|\textbf{f}_\textbf{n}(\textbf{x})|\in \textbf{R}$. We define
\[
\nco_{\sigma,\textbf{R}}[\textbf{f}_\textbf{n}\:;\:\textbf{n}\geq \textbf{1}]:=\left[ \sum_{\textbf{n}\geq \textbf{1}} \textbf{r}_\textbf{n} \textbf{f}_\textbf{n} \:;\: \{\textbf{r}_\textbf{n}\}\subset\textbf{R}^+,\:\sum_{\textbf{n}\geq \textbf{1}} \textbf{r}_\textbf{n}=\textbf{1}  \right].
\]

Notice that, due to the conditional boundedness of $\textbf{f}_\textbf{n}(\textbf{x})$, we have  $\sum_{\textbf{n}\geq 1} \textbf{r}_\textbf{n} \textbf{f}_\textbf{n}(\textbf{x})\in\textbf{R}$ for all $\textbf{x}$ in $\textbf{E}$. 



Given a conditional function $\textbf{f}:\textbf{C}\rightarrow\textbf{R}$, we denote the conditional supremum of $\textbf{f}$ on $\textbf{C}$ by $\nsup_\textbf{C}\textbf{f}:=\nsup[\textbf{f}(\textbf{x})\:;\: \textbf{x}\in\textbf{C}]$. 

A key piece of the conditional version of James' theorem will be the following result, which is a generalization
of the sup-limsup theorem of Simons \cite[Theorem 3]{key-29}. The proof of this result is an adaptation to a conditional setting of the proof of \cite[Corollary 10.6]{key-23} (which also relies on Lemma 10.4 and Theorem 10.5 of \cite{key-23}). For saving space, since this adaptation does not have any surprising element, we have omitted the proof.


\begin{thm}\label{thm: Simons' sup-limsup}
[Conditional version of Simons' sup-limsup theorem]
Let $\textbf{E}$ be a conditional set, and let $\left\{\textbf{f}_\textbf{n}\right\}_{\textbf{n}\in\textbf{N}}$ be a conditional sequence of conditional functions $\textbf{f}_\textbf{n}:\textbf{E}\rightarrow\textbf{R}$ such that for each $\textbf{x}\in\textbf{E}$ it holds $\nsup [|\textbf{f}_\textbf{n}(\textbf{x})|\:;\: \textbf{n}\in\textbf{N}]\in\textbf{R}$. Suppose that $\textbf{C}$ is a conditional subset of $\textbf{E}$ such that 
for every conditional function $\textbf{g}\in{\nco_{\sigma,\textbf{R}}}\left[\textbf{f}_\textbf{n} \:;\: \textbf{n}\geq \textbf{1} \right]$ there exists $\textbf{z}\in\textbf{C}$ with $\textbf{g}(\textbf{z})=\nsup_\textbf{E}\textbf{g}$. Then,
\[
\nsup_\textbf{E} \left(\nlimsup_\textbf{n} \textbf{f}_\textbf{n}\right)= \nsup_\textbf{C}\left(\nlimsup_\textbf{n} \textbf{f}_\textbf{n}\right).
\]
\end{thm}

Let see some necessary notions:

\begin{itemize}
	\item \cite{key-7} A conditional sequence $\{\textbf{x}_\textbf{n}\}$ in a conditional normed space $(\textbf{E},\nVert\cdot\nVert)$ is said to be  Cauchy, if for every $\textbf{r}\in\textbf{R}^{++}$ there exists $\textbf{n}_\textbf{r}\in\textbf{N}$ such that $\nVert \textbf{x}_\textbf{p} - \textbf{x}_\textbf{q}\nVert\leq \textbf{r}$ for all $\textbf{p},\textbf{q}\geq\textbf{n}_\textbf{r}$. 
	\item \cite{key-7} A conditional normed space is said to be Banach, if every conditional Cauchy sequence converges.
\item Let $(\textbf{E},\nVert\cdot\nVert)$ be a conditional Banach space, and let $\{\textbf{x}^*_\textbf{n}\}$ be  a conditional sequence in $\textbf{E}^*$. We define $\textbf{L}\{\textbf{x}_\textbf{n}^*\}$ the conditional set of conditional cluster points of $\{\textbf{x}_\textbf{n}^*\}$ in the conditional topology $\sigma(\textbf{E}^*,\textbf{E})$, i.e. $\textbf{x}^*\in\textbf{L}\{\textbf{x}_\textbf{n}^*\}$ if, and only if, for every conditional neighborhood $\textbf{U}$ of $\textbf{x}^*$ on $1$ and every $\textbf{n}\in\textbf{N}$ there is $\textbf{x}_\textbf{m}^*\in\textbf{U}$ with $\textbf{m}\geq\textbf{n}$. 
\item Let $\textbf{E}$ be a conditional vector space, and let $\{\textbf{x}_\textbf{n}\}$ and $\{\textbf{y}_\textbf{n}\}$ be conditional sequences in $\textbf{E}$. Then, $\{\textbf{y}_\textbf{n}\}$ is said to be a \textit{conditional convex block sequence} of $\{\textbf{x}_\textbf{n}\}$, if there exists a sequence of conditional natural numbers $\textbf{1}=\textbf{n}_1<\textbf{n}_2<...$ and a conditional sequence $\{\textbf{r}_\textbf{n}\}$ of conditional real number with $\textbf{0}\leq\textbf{r}_\textbf{n}\leq \textbf{1}$ for all $\textbf{n}\in\textbf{N}$, in such a way that
\begin{equation}
\begin{array}{cccc}
\sum_{\textbf{n}_k\leq\textbf{i}<\textbf{n}_{k+1}}\textbf{r}_\textbf{i}=\textbf{1} & \textnormal{ and } & \sum_{\textbf{n}_k\leq\textbf{i}<\textbf{n}_{k+1}}\textbf{r}_\textbf{i} \textbf{x}_\textbf{i}=\textbf{y}_\textbf{k} &\textnormal{ for each }k\in\N.
\end{array}
\end{equation}

For a conditional Banach space $\textbf{E}$, its conditional dual unit ball $\textbf{B}_{\textbf{E}^*}$ is said to be conditionally weakly-$*$ \textit{convex block compact} provided that each conditional sequence $\{\textbf{x}_\textbf{n}\}$ in $\textbf{B}_{\textbf{E}^*}$
 has a conditionally convex block weakly-$*$ convergent sequence.
\end{itemize}


\begin{prop}
For $1\leq p\leq \infty$, $(\textbf{L}^p_\F(\EE),\nVert\cdot\nVert_p)$ is a conditional Banach space. 
\end{prop}
\begin{proof}

In \cite{key-10}, it is shown that the $L^0$-normed module $({L}^p_\F(\EE),\Vert\cdot|\F\Vert_p)$ is complete in the sense that every Cauchy net converges in $L^p_\F(\EE)$. Now, let $\{\textbf{x}_\textbf{n}\}$ be a conditional Cauchy sequence in $\textbf{L}^p_\F(\EE)$. Then, we can consider the stable family $\{x_n\}_{n\in L^0(\N)}$. We have that $L^0(\N)$ is  directed upwards, and therefore $\{x_n\}_{n\in L^0(\N)}$ is a net indexed by $L^0(\N)$. Furthermore, since  $\{\textbf{x}_\textbf{n}\}$ is conditionally Cauchy, it follows that  $\{x_n\}_{n\in L^0(\N)}$ is Cauchy. Finally, since $L^p_\F(\EE)$ is complete, $\{x_n\}_{n\in L^0(\N)}$ converges to some $x_0\in{L}^p_\F(\EE) $. If follows that $[\textbf{x}_\textbf{n}]$ conditionally converges to $\textbf{x}_0$.

\end{proof}

Let $(\textbf{E},\nVert\cdot\nVert)$ be a conditional Banach space. Let us denote $\textbf{E}^{**}:=(\textbf{E}^*)^*$. For each $\textbf{x}\in\textbf{E}$, let us define the conditional function $\textbf{j}_\textbf{x}:\textbf{E}^*\rightarrow\textbf{R}$ with $\textbf{j}_\textbf{x}(\textbf{x}^*)=\textbf{x}^*(\textbf{x})$. Then, in \cite[Theorem 3.5]{key-38} was proved that $\textbf{j}:\textbf{E}\rightarrow\textbf{E}^{**}$ with $\textbf{j}(\textbf{x}):=\textbf{j}_\textbf{x}$ is a conditional isometry, i.e. $\Vert \textbf{x}\Vert=\Vert\textbf{j}(\textbf{x})\Vert$ for all $\textbf{x}\in\textbf{E}$, which is referred to as \textit{natural conditional embedding}.  

A conditional version of the classical Eberlein-\v{S}mulian theorem was provided in \cite[Theorem 4.7]{key-38}. A simple variation of this result is the following (the proof is provided in the Appendix):

\begin{thm}
\label{thm: EberleinSmulianII}
Let $(\textbf{E},\nVert\cdot\nVert)$ be a conditional normed space and let $\textbf{K}\sqsubset\textbf{E}$ be on $1$ and  conditionally bounded. Let $\textbf{j}:\textbf{E}\rightarrow\textbf{E}^{**}$ be natural conditional embedding, and suppose that
\[
\begin{array}{cc}
\overline{\textbf{j}(\textbf{K})}^{\sigma(\textbf{E}^{**},\textbf{E}^*)}\sqcap\textbf{j}(\textbf{E})^\sqsubset & \textnormal{ is on }1.
\end{array}
\]
Then, there exists a conditional sequence $\{\textbf{x}_\textbf{n}\}$ in $\textbf{K}$ with a conditional cluster point $\textbf{x}^{**}\in\textbf{E}^{**}\sqcap\textbf{j}(\textbf{E})^\sqsubset$ with respect the conditional $\sigma(\textbf{E}^{**},\textbf{E}^*)$-topology.
\end{thm}



We have the following result:

\begin{lem}
\label{lem: weakCompact}
Let $(\textbf{E},\nVert\cdot\nVert)$ be a conditional Banach space. Suppose that the conditional dual unit ball of $\textbf{E}$ is conditionally weakly-$*$ convex block compact and that $\textbf{K}$ is a  conditional subset of $\textbf{E}$ on $1$ which is conditionally bounded. Then $\textbf{K}$ is conditionally weakly relatively compact if, and only if, each conditional sequence $\{\textbf{x}^*_\textbf{n}\}$ in $\textbf{E}^*$ such that $\nlim_{\textbf{n}} \textbf{x}^*_\textbf{n}=\textbf{0}$ with $\sigma(\textbf{E}^*,\textbf{E})$, also satisfies that $\nlim_{\textbf{n}} \textbf{x}^*_\textbf{n}=\textbf{0}$ with $\sigma(\textbf{E}^*,\overline{\textbf{j}(\textbf{K})}^{\sigma(\textbf{E}^{**},\textbf{E}^*)})$, where $\textbf{j}:\textbf{E}\rightarrow\textbf{E}^{**}$ is the natural conditional embedding. 
\end{lem}
\begin{proof}
Let us denote $\overline{\textbf{K}}^{\omega^*}=\overline{\textbf{j}(\textbf{K})}^{\sigma(\textbf{E}^{**},\textbf{E}^*)}$. If $\textbf{K}$ is conditionally  weakly relatively compact, then $\overline{\textbf{K}}^{\omega^*}\sqsubset \textbf{j}(\textbf{E})$ and, in view of \cite[Lemma 4.2]{key-38}, we are done.  

Conversely, let us define
\[
b:=\vee\left\{a\in\A\:;\: \overline{\textbf{K}}^{\omega^*}|a\sqsubset \textbf{j}(\textbf{E})|a  \right\}.
\]

If $b=1$, we are done. If not, we can consistently argue on $b^c$. Thus, we can suppose $b=0$ w.l.g. If so,  Theorem \ref{thm: EberleinSmulianII} guarantees the existence of a conditional sequence $\{\textbf{x}_\textbf{n}\}$ in $\textbf{K}$ with a conditional $\sigma(\textbf{E}^{**},\textbf{E}^*)$-cluster point $\textbf{x}_0^{**}\in\textbf{E}^{**}\sqcap\textbf{j}(\textbf{E})^{\sqsubset}$. We can now apply the conditional version of the separation Hahn-Banach theorem (see \cite[Theorem 5.5]{key-7})  to obtain $\textbf{x}^{***}\in\textbf{B}_{\textbf{E}^{***}}$ such that
\begin{equation}
\label{eqI}
\begin{array}{ccc}
\textbf{x}^{***}(\textbf{x}^{**}_0)\in[\textbf{0}]^\sqsubset & \textnormal{ and }&\textbf{x}^{***}(\textbf{j}(\textbf{x}))=\textbf{0}\textnormal{ for all }\textbf{x}\in\textbf{E}.
\end{array}
\end{equation}

For $\textbf{n}\in\textbf{N}$, let us define the conditional set
\[
\textbf{U}_\textbf{n}:=\left[\textbf{z}^{***}\in\textbf{E}^{***} \:;\: \textbf{z}^{***}(\textbf{x}^{**}_0)\leq \frac{\textbf{1}}{\textbf{n}},\textbf{z}^{***}(\textbf{j}(\textbf{x}_\textbf{k}))\leq\frac{\textbf{1}}{\textbf{n}}\textnormal{ for all }\textbf{k}\in \textbf{N}\textnormal{ with } \textbf{k}\leq\textbf{n} \right].
\]

Then $[\textbf{U}_\textbf{n}\:;\:\textbf{n}\in\textbf{N}]$ is a conditional neighborhood base of $\textbf{0}\in\textbf{E}^{***}$ for the conditional topology $\sigma(\textbf{E}^{***},[\textbf{x}^{**}_0,\textbf{j}(\textbf{x}_\textbf{n})\:;\:\textbf{n}\in\textbf{N}])$.

Now, the conditional version of Goldstine's theorem (see \cite[Theorem 3.6]{key-38}) claims that $\textbf{j}^*(\textbf{B}_{\textbf{E}^{*}})$ is conditionally $\sigma(\textbf{E}^{***},\textbf{E}^{**})$-dense in $\textbf{B}_{\textbf{E}^{***}}$ (where $\textbf{j}^*:\textbf{E}^*\rightarrow \textbf{E}^{***}$ is the natural conditional embedding of $\textbf{E}^*$). In particular, for each $n\in\N$ there exists $\textbf{x}^*_n\in\textbf{B}_{\textbf{E}^*}$ with $\textbf{j}^*(\textbf{x}^*_n)\in\textbf{U}_\textbf{n}$. For an arbitrary $\textbf{n}\in\textbf{N}$ with $n=\sum_k n_k|a_k$, we define $x^*_n:=\sum_k x^*_{n_k}|a_k$. 

Then, we obtain a conditional sequence $\{\textbf{x}^*_\textbf{n}\}$ in $\textbf{B}_{\textbf{E}^*}$ such that 

\begin{equation}
\label{eqII}
\sigma(\textbf{E}^{***},[\textbf{x}^{**}_0,\textbf{j}(\textbf{x}_\textbf{n})\:;\:\textbf{n}\in\textbf{N}])-\nlim_\textbf{n} \textbf{j}^*(\textbf{x}^*_\textbf{n})=\textbf{x}^{***}.
\end{equation}

Since $\textbf{B}_{\textbf{E}^*}$ is conditionally convex block $\sigma(\textbf{E}^*,\textbf{E})$-compact, there exists a conditionally  convex block compact sequence $\{\textbf{y}^*_\textbf{n}\}$ of $\{\textbf{x}^*_\textbf{n}\}$ and $\textbf{x}^*_0\in\textbf{B}_{\textbf{E}^*}$ such that $\sigma(\textbf{E}^*,\textbf{E})-\nlim_\textbf{n} \textbf{y}^*_\textbf{n}=\textbf{x}^*_0$.

Then, by assumption, we have $\sigma(\textbf{E}^*,\overline{\textbf{K}}^{\omega^*})-\nlim_\textbf{n} \textbf{y}^*_\textbf{n}=\textbf{x}^*_0$, and so
\begin{equation}
\label{eqIII}
\sigma(\textbf{E}^{***},[\textbf{x}^{**}_0,\textbf{j}(\textbf{x}_\textbf{n})\:;\:\textbf{n}\in\textbf{N}])-\nlim_\textbf{n} \textbf{j}^*(\textbf{y}^*_\textbf{n})=\textbf{j}^*(\textbf{x}^*_0).
\end{equation}

Finally, it follows from (\ref{eqI}), (\ref{eqII}) and (\ref{eqIII})
\[
\textbf{x}^{**}_0(\textbf{x}^{*}_0)=\nlim_\textbf{n} \textbf{x}^{**}_0(\textbf{y}^*_\textbf{n})=\nlim_\textbf{n} \textbf{x}^{**}_0(\textbf{x}^*_\textbf{n})=\textbf{x}^{***}(\textbf{x}^{**}_0)\in [\textbf{0}]^\sqsubset,
\]
but for each $\textbf{k}\in\textbf{N}$,
\[
\textbf{x}^{*}_0(\textbf{x}_\textbf{k})=\nlim_\textbf{n} \textbf{y}^*_\textbf{n}(\textbf{x}_\textbf{k})=\nlim_\textbf{n} \textbf{x}^*_\textbf{n}(\textbf{x}_\textbf{k})=\textbf{x}^{***}(\textbf{j}(\textbf{x}_\textbf{k}))=\textbf{0},
\] 
which is a contradiction, because $\textbf{x}^{**}_0$ is a conditional $\sigma(\textbf{E}^{**},\textbf{E}^*)$-cluster point of $\{\textbf{x}_\textbf{n}\}$.
\end{proof}

\begin{thm}\label{thm: Rainwater}[Conditional and unbounded version of Rainwater–Simons’ theorem]
Let $(\textbf{E},\nVert\cdot\nVert)$ be a conditional normed space and let $\textbf{C},\textbf{B}$ be conditional subsets of $\textbf{E}^*$ on $1$ with $\textbf{B}\sqsubset\textbf{C}$. Suppose that $\{\textbf{x}_\textbf{n}\}$ is a conditionally bounded sequence in $\textbf{E}$ such that
\[
\textnormal{for every } \textbf{x} \in\nco_{\sigma, \textbf{R}} \left[\textbf{x}_\textbf{n}\right] \textnormal{ there exists } \textbf{b}^* \in\textbf{B} \textnormal{ with } \textbf{b}^*(\textbf{x}) =\nsup \left[\textbf{x}^*(\textbf{x})\:;\: \textbf{x}^*\in\textbf{C}\right], 
\]
then,
\[
\underset{\textbf{x}^*\in\textbf{B}}\nsup \underset{\textbf{n}}\nlimsup \textbf{x}^*(\textbf{x}_\textbf{n}) = \underset{\textbf{x}^*\in\textbf{C}}\nsup \underset{\textbf{n}}\nlimsup \textbf{x}^*(\textbf{x}_\textbf{n}). 
\]
As a consequence, if there exists a conditional sequence $\{\textbf{y}_\textbf{n}\}$ such that $\sigma(\textbf{E},\textbf{C})-\nlim_\textbf{n}\textbf{y}_\textbf{n}=\textbf{0}$ and so that
\[
\textnormal{for every } \textbf{x} \in\nco_{\sigma, \textbf{R}} \left[\textbf{x}_\textbf{n}+\textbf{y}_\textbf{n}\right]\sqcup\nco_{\sigma, \textbf{R}} \left[-\textbf{x}_\textbf{n}+\textbf{y}_\textbf{n}\right]
\]
\[
\textnormal{ there exists } \textbf{b}^* \in\textbf{B} \textnormal{ with } \textbf{b}^*(\textbf{x}) =\nsup \left[\textbf{x}^*(\textbf{x})\:;\: \textbf{x}^*\in\textbf{C}\right], 
\]
then
\[
\sigma (\textbf{E},\textbf{B})-\underset{\textbf{n}}\nlim \textbf{x}_\textbf{n} =\textbf{0} \textnormal{ implies } \sigma (\textbf{E},\textbf{C})-\underset{\textbf{n}}\nlim \textbf{x}_\textbf{n} =\textbf{0}. 
\]

\end{thm}
\begin{proof}

First part is a consequence of Theorem \ref{thm: Simons' sup-limsup}.  

For the second part, let us fix $\textbf{x}^*\in\textbf{C}$, then
 
\[
\nlimsup_\textbf{n} \textbf{x}^*(\textbf{x}_\textbf{n})=\nlimsup_\textbf{n} \textbf{x}^*(\textbf{x}_\textbf{n}+\textbf{y}_\textbf{n}) \leq \sigma (\textbf{E},\textbf{B})-\underset{\textbf{n}}\nlim (\textbf{x}_\textbf{n}+\textbf{y}_\textbf{n})=\sigma (\textbf{E},\textbf{B})-\underset{\textbf{n}}\nlim \textbf{x}_\textbf{n} =\textbf{0}. 
\] 
On the other hand, 
\[
\nliminf_\textbf{n} \textbf{x}^*(\textbf{x}_\textbf{n})=-\nlimsup_\textbf{n} \textbf{x}^*(-\textbf{x}_\textbf{n})=-\nlimsup_\textbf{n} [\textbf{x}^*(-\textbf{x}_\textbf{n}+\textbf{y}_\textbf{n})]\geq 
\]
\[
 \geq \sigma (\textbf{E},\textbf{B})-\underset{\textbf{n}}\nlim (\textbf{x}_\textbf{n}-\textbf{y}_\textbf{n} )=\sigma (\textbf{E},\textbf{B})-\underset{\textbf{n}}\nlim \textbf{x}_\textbf{n}=\textbf{0}.  
\]
Then, $\nlim_\textbf{n} \textbf{x}^*(\textbf{x}_\textbf{n})=\textbf{0}$, and, since $\textbf{x}^*$ is arbitrary, we conclude that $\sigma (\textbf{E},\textbf{C})-\underset{\textbf{n}}\nlim \textbf{x}_\textbf{n} =\textbf{0}$.
\end{proof}
Let us introduce some terminology:

\begin{defn}
Let $\textbf{E}$ be a conditional vector space. The conditional effective domain of a conditional function $\textbf{f}:\textbf{E}\rightarrow\overline{\textbf{R}}$ is denoted by $\dom(\textbf{f}):=[\textbf{x}\:;\:\textbf{f}(\textbf{x})\in\textbf{R}]$. The conditional epigraph of $\textbf{f}$ is denoted by $\epi(\textbf{f}):=[(\textbf{x},\textbf{r})\in\textbf{E}\Join\textbf{R}\:;\:\textbf{f}(\textbf{x})\leq\textbf{r}]$. The conditional function $\textbf{f}$ is said to be proper if $\textbf{f}(\textbf{x})>-\infty$ for all $\textbf{x}\in\textbf{E}$ and there exists some $\textbf{x}\in\dom(\textbf{f})$.
\end{defn}

\begin{thm}\label{thm: James}[Conditional unbounded version of James' Theorem]
Let $(\textbf{E},\nVert\cdot\nVert)$ be a conditional Banach space such that its conditional dual unit ball $\textbf{B}_{\textbf{E}^*}$ is conditionally $\omega^*$-convex block compact and let $\textbf{f}: \textbf{E}\rightarrow \overline{\textbf{R}}$ be a conditionally proper  function such that 
\[
\textnormal{ for all } \textbf{x}^* \in\textbf{E}^*\textnormal{, there is a }\textbf{x}_0\in\textbf{E}\textnormal{ so that }  \textbf{x}^*(\textbf{x}_0) - \textbf{f}(\textbf{x}_0)=\nsup \left[ \textbf{x}^*(\textbf{x}) - \textbf{f}(\textbf{x}) \:;\: \textbf{x}\in\textbf{E}\right], 
\]
then for every conditional real number $\textbf{y}$, the conditional sublevel set $\textbf{V}_\textbf{f}(\textbf{y}):=\left[\textbf{z}\:;\: \textbf{f}(\textbf{z})\leq\textbf{y}\right]$ is conditionally weakly relatively compact.
\end{thm}
\begin{proof}
Let us fix a conditional real number $\textbf{y}_0$. Let us put $\textbf{K}:=\textbf{V}_{\textbf{f}} (\textbf{y}_0)$. Suppose that $\textbf{K}$ is on $a$. If $a=0$, we are done. If $a>0$, we can suppose $a=1$ by consistently arguing on $a$. 

The conditional uniform boundedness principle (\cite[Theorem 3.4]{key-38})  and the optimization assumption on $\textbf{f}$ imply that $\textbf{K}$ is conditionally bounded. In order to obtain the conditional relative weak compactness of $\textbf{K}$ we apply Lemma \ref{lem: weakCompact}. Thus, let us consider a conditional sequence $\{\textbf{x}^*_\textbf{n}\}$ in $\textbf{E}^*$ such that $\sigma(\textbf{E}^{*},\textbf{E})-\nlim_{\textbf{n}} \textbf{x}^*_\textbf{n} = \textbf{0}$  and let us show that ${\sigma ( \textbf{E}^{*}, \overline{\textbf{K}}^{\omega^*}  )}-\nlim_{\textbf{n}} \textbf{x}^*_\textbf{n} = \textbf{0}$.

Let us fix $(\textbf{x}^*,\textbf{l})\in \textbf{E}^* \Join \textbf{R}^{--}$. By assumption, we have $\textbf{x}_0\in\textbf{E}$ such that 
\[
\textbf{x}^*(\textbf{x}_0){\textbf{l}}^{-1} - \textbf{f}(\textbf{x}_0)=\nsup \left[ \textbf{x}^*(\textbf{x}){\textbf{l}}^{-1} - \textbf{f}(\textbf{x}) \:;\: \textbf{x}\in \textbf{E}\right]. 
\]

Let us define $\textbf{B}:=\epi(\textbf{f})\sqsubset \textbf{C}:=\overline{\textbf{B}}^{\sigma(\textbf{E}^{**}\Join\textbf{R},\textbf{E}^{*}\Join\textbf{R})}$.

We claim that 
\[
\nsup\left[ \langle (\textbf{x}^*, \textbf{l}),(\textbf{x},\textbf{y}) \rangle \:;\: (\textbf{x},\textbf{y})\in\textbf{B}\right]=\langle (\textbf{x}^*, \textbf{l}),(\textbf{x}_0,\textbf{f}(\textbf{x}_0)) \rangle,
\] 
where $\langle (\textbf{x}^*, \textbf{l}),(\textbf{x},\textbf{y}) \rangle := \textbf{x}^*(\textbf{x})+\textbf{l}\textbf{y}$ for $(\textbf{x}^*,\textbf{l},\textbf{x},\textbf{y})\in\textbf{E}^*\Join\textbf{R}\Join\textbf{E}\Join\textbf{R}$.

Indeed, define $\textbf{z}^*:=\textbf{x}^* / \textbf{l}$. Then,
\[
\nsup\left[ \langle (\textbf{x}^*, \textbf{l}),(\textbf{x},\textbf{y}) \rangle \:;\: (\textbf{x},\textbf{y})\in\textbf{B}\right\}\leq -\textbf{l} \underset{(\textbf{x},\textbf{y})\in \epi(\textbf{g})} \nsup \left( \textbf{z}^*(\textbf{x}) - \textbf{y} \right) \leq
\] 
\[
\leq -\textbf{l} \underset{\textbf{x}\in\dom(\textbf{f})} \nsup \left( \textbf{z}^*(\textbf{x}) - \textbf{f}(\textbf{x}) \right) \leq -\textbf{l} \underset{\textbf{x}\in\textbf{E}} \nsup \left( \textbf{z}^*(\textbf{x}) - \textbf{f}(\textbf{x}) \right)= 
\]
\[
=-\textbf{l}  \left( \textbf{z}^*(\textbf{x}_0) - \textbf{f}(\textbf{x}_0) \right) = \textbf{x}^*(\textbf{x}_0)+\textbf{l} \textbf{f}(\textbf{x}_0).
\]
Note that $\textbf{x}_0\in\dom (\textbf{f})$ as $\textbf{f}$ is conditionally proper.

Moreover, since $\langle (\textbf{x}^*, \textbf{l}), \cdot \rangle :  \textbf{E}^{**}\Join \textbf{R} \rightarrow \textbf{R}$ is conditionally $\sigma(E^{**}\Join \textbf{R}, \textbf{E}^*\Join \textbf{R})$-continuous, it holds

\[
\nsup\left[ \langle (\textbf{x}^*, \textbf{l}),(\textbf{x},\textbf{y}) \rangle \:;\: (\textbf{x},\textbf{y})\in\textbf{C}\right]=
\]
\[
=\nsup\left[ \langle (\textbf{x}^*, \textbf{l}),(\textbf{x},\textbf{y}) \rangle \:;\: (\textbf{x},\textbf{y})\in\textbf{B}\right]=\textbf{x}^*(\textbf{x}_0)+\textbf{l} \textbf{f}(\textbf{x}_0).
\]

Now, let us consider the conditionally bounded sequence

\[
\left\{ \left( \textbf{x}^*_\textbf{n} , -\frac{\textbf{1}}{\textbf{n}}\right) \right\}_{\textbf{n}\in\textbf{N}} \textnormal{ in }\textbf{E}^*\Join\textbf{R}.
\]

It is clear that $\sigma(\textbf{E}^*\Join\textbf{R},\textbf{B})-\nlim_\textbf{n} (\textbf{x}^*_\textbf{n},\textbf{0})=\textbf{0}$ and $\sigma(\textbf{E}^*\Join\textbf{R},\textbf{C})-\nlim_\textbf{n} (\textbf{0},\frac{-1}{\textbf{n}})=\textbf{0}$. Then, in view of Theorem \ref{thm: Rainwater}, we obtain 
\[
\sigma(\textbf{E}^*\Join\textbf{R},\textbf{C})-\nlim_\textbf{n}  (\textbf{x}^*_\textbf{n},\textbf{0})=\textbf{0},
\]
and, since $\overline{\textbf{K}}^{\omega^*}\Join [\textbf{0}]\sqsubset\textbf{C}$, it follows that $\sigma(\textbf{E}^*,\overline{\textbf{K}}^{\omega^*})-\nlim_\textbf{n}\textbf{x}^*_\textbf{n}=\textbf{0}$, and the proof is complete.

\end{proof}

Finally, as a consequence of theorem above we show a conditional version of the classical  James' compactness Theorem —no conditional function is involved— for conditional Banach spaces with conditionally $\omega^*$-convex block compact dual unit ball.     

\begin{thm}\label{thm: JamesII}[Conditional version of James' Theorem]
Let $(\textbf{E},\nVert\cdot\nVert)$ be a conditional Banach space such that its conditional dual unit ball $\textbf{B}_{\textbf{E}^*}$ is conditionally $\omega^*$-convex block compact and let $\textbf{K}$ be on $1$, conditionally weakly closed subset of $\textbf{E}$.  The conditional set $\textbf{K}$ is conditionally weakly compact if, and only if, for each $\textbf{x}^*\in\textbf{E}^*$ there is $\textbf{x}_0\in\textbf{K}$ such that $\textbf{x}^*(\textbf{x}_0)=\nsup_{\textbf{x} \in\textbf{K}} \textbf{x}^*(\textbf{x})$.   
\end{thm}
\begin{proof}
Let $\textbf{K}$ be on $1$ and  conditionally weakly compact. For each $\textbf{x}^*\in\textbf{E}^*$, due to Proposition 3.26 of \cite{key-7}, we know that $\textbf{x}^*(\textbf{K})$ is a conditionally compact subset of $\textbf{R}$. Therefore, by the conditional version of the Heine-Borel theorem (see \cite[Theorem 5.5.8]{key-40}), $\textbf{x}^*(\textbf{K})$ is conditionally closed and bounded. From this fact, it can be showed that $\textbf{x}^*$ attains a conditional maximum on $\textbf{K}$.

For the converse, let us consider the conditional function $\textbf{f}:\textbf{E}\rightarrow\overline{\textbf{R}}$ defined as follows: For $\textbf{x}\in\textbf{E}$ we define the stable function $f(x):=1|b + \infty|b^c$, with $b:=\vee\left\{a\in\A\:;\: x|a\in\textbf{K}\right\}$. Then, $\textbf{V}_\textbf{f}(\textbf{1})=\textbf{K}$ and $\textbf{f}$ satisfies the hypothesis of Theorem \ref{thm: James}. We conclude that $\textbf{K}$ is conditionally weakly compact.   
\end{proof}

The following example exhibits how the latter result applies:

\begin{example}
Let $(\Omega,\EE,\PP)$ be a probability space and $\F\subset\EE$ a sub-$\sigma$-algebra. Let us define $\mathcal{P}_\F:=\left\{Q \ll \PP\:;\:Q|_{\F}=\PP|_{\F}\right\}$. This set defines a conditional set. Indeed, for $a\in\A$ and $Q\in\mathcal{P}_\F$ we define $Q|a:=Q(\cdot|A)$, where $A$ is some representative of the equivalence class $a$ (note that the definition does not depend on the choice of $A$). For a given  countable family $\{Q_k\}$ in $\mathcal{P}_\F$ and $\{a_k\}\in p(1)$ we take $\sum_{k} Q_k|a_k:=\sum_k Q_k(\cdot|A_k)Q_k(A_k)$ where $A_k\in\F$ is a representative of $a_k$ (again there is no dependence on representatives). An easy verification shows that $\sum_{k} Q_k|a_k\in \mathcal{P}_\F$. 

Let us take a stable subset $S$ of $\mathcal{P}_\F$.

It can be shown by inspection that the set of Radon-Nykodim derivatives $K=\left\{\frac{d Q}{d\PP}\:;\: Q\in S \right\}$ is a stable subset of $L^1_\F(\EE)$. Let us suppose that $\textbf{K}$ is conditionally bounded and weakly closed. Then we claim that, for each $x\in L^\infty(\EE)$, the supremum
\begin{equation}
\label{sup1}
\esssup \left\{ \E_Q[x|\F] \:; \: Q\in S\right\}
\end{equation}
is attained if, and only if, $\textbf{K}$ is conditionally weakly compact.

For each $Q\in\mathcal{P}_\F$ we have $\E_Q[x|\F]=\E_\PP[x \frac{d Q}{d\PP}|\F]$.  Thus we have the following equality: 
\begin{equation}
\label{sup2}
\left\{ \E_Q[x|\F] \:; \: Q\in S\right\}=\left\{ \E_\PP\left[x y|\F\right] \:;\: y\in K\right\}.
\end{equation}
Besides, for $Q\in\mathcal{P}_\F$ we have $\E_\PP[\frac{d Q}{d\PP}|\F]=1$. This means that $\textbf{K}$ is also conditionally bounded. Also, by Example \ref{exam: Lp}, we know that $(\textbf{L}^1_\F(\EE))^*=\textbf{L}^\infty_\F(\EE)$. Further, we will show in Lemma \ref{lem: WCG} that $\textbf{L}^\infty_\F(\EE)$ has conditionally $\omega^*$-convex block compact unit ball. Then, by Corollary \ref{thm: JamesII}, we conclude that for each $x\in L^\infty_\F(\EE)$, the supremum (\ref{sup1}) is attained if, and only if, $\textbf{K}$ is conditionally weakly compact.

Finally, since every $x\in L^\infty_\F(\EE)$ is of the form $x=y x_0$ with $y\in L^0_+(\F)$ and $x_0\in L^\infty(\EE)$, it follows that the supremum (\ref{sup1}) is attained for every $x\in L^\infty(\EE)$ if, and only if, $\textbf{K}$ is conditionally weakly compact.

\end{example}

%
%

\subsection{An application to conditional risk measures}

 In what follows we will show an application of the conditional James' compactness theorem. Namely, we will obtain a conditional version of the so-called Jouini-Schachermayer-Touzi theorem for  convex risk measures in a conditional setting. For this purpose, it is important to keep in mind the duality relation shown in Example \ref{exam: Lp}. Hereafter, we fix $(\Omega,\mathcal{E},\PP)$ a probability space and $\F$ a sub-$\sigma$-algebra of $\mathcal{E}$. Since $\F,\mathcal{E}$ are fixed,  throughout we will use the notation $\textbf{L}^p:=\textbf{L}^p_\F(\mathcal{E})$.

We will study the following type of risk measures, which are defined in the classical setting:

\begin{defn}
\label{defn: condRiskMeasure}
A function $\rho : L_{\mathcal{F}}^\infty\left(\mathcal{E}\right) \rightarrow L^0(\F)$ is called a convex risk measure if $\rho$ is :
\begin{enumerate}
	\item monotone, i.e. if $x\leq y$ then $\rho(x)\geq\rho(y)$; 
  \item $L^0(\F)$-cash invariant, i.e. if $(\eta,x) \in L^0(\F)\times{L}_{\mathcal{F}}^\infty\left(\mathcal{E}\right)$, then $\rho(x+\eta)=\rho(x)-\eta$;
  \item convex, i.e. $\rho(r x + (1-r) y)\leq r \rho(x) + (1-r) \rho(y)$ for all $r\in \R$ with $0\leq r \leq 1$ and $x,y\in L_{\mathcal{F}}^{\infty}\left(\mathcal{E}\right)$.
	%

Its Fenchel conjugate is defined by:
\[
\begin{array}{cc}
\rho^*(y):=\esssup\{\E_\PP[x y|\F]-\rho(x) \:;\: x\in L^\infty_\F(\EE) \} & \textnormal{ for }y\in L^1_\F(\EE).
\end{array}
\]
\end{enumerate}
\end{defn}

Filipovic et al.\cite{key-10} proposed to study risk measures that are $L^0(\F)$-convex. We would like to emphasize that we are considering the weaker assumption of convexity in the traditional sense.  Nevertheless, we have the following result:


\begin{lem}
\label{lem: cont}

Any convex risk measure $\rho:L^\infty_\F(\EE)\rightarrow L^0(\F)$ generates a conditional function $\boldsymbol{\rho} : \textbf{L}^\infty\rightarrow \textbf{R}$, which is:
\begin{enumerate}
 
	\item conditionally monotone, i.e. if $\textbf{x}\leq \textbf{y}$ then $\boldsymbol\rho(\textbf{x})\geq\boldsymbol\rho(\textbf{y})$;
	\item conditionally cash invariant, i.e. if $(\textbf{r},\textbf{x}) \in \textbf{R}\Join\textbf{L}^{p}$, then $\boldsymbol\rho\textbf{(x}+\textbf{r})=\boldsymbol\rho(\textbf{x})-\textbf{r}$;
	\item conditionally convex, i.e. $\boldsymbol\rho(\textbf{r} \textbf{x} + (\textbf{1}-\textbf{r}) \textbf{y})\leq \textbf{r} \boldsymbol\rho(\textbf{x}) + (\textbf{1}-\textbf{r}) \boldsymbol\rho(\textbf{y})$ for all $\textbf{r}\in \textbf{R}$ with $\textbf{0}\leq \textbf{r} \leq \textbf{1}$ and $\textbf{x},\textbf{y}\in \textbf{L}^{\infty}$;
	\item conditionally Lipschitz continuous, i.e. $|\boldsymbol\rho(\textbf{x}) − \boldsymbol\rho(\textbf{y})| \leq\boldsymbol\Vert \textbf{x} − \textbf{y}\textbf{}\Vert_\infty$, for all $\textbf{x},\textbf{y}\in\textbf{L}^\infty$.
\end{enumerate}

Moreover, $\rho^*:L^1_\F(\mathcal{E})\rightarrow L^0(\F)$ is also stable and generates a conditional function $\boldsymbol\rho^*:\textbf{L}^1\rightarrow \textbf{R}$, which satisfies:
\[
\begin{array}{cc}
\boldsymbol{\rho}^*(\textbf{y}):=\nsup\left[\E_\PP[\textbf{x} \textbf{y}|\F]-\boldsymbol{\rho}(\textbf{x}) \:;\: \textbf{x}\in\textbf{L}^\infty \right] & \textnormal{ for }\textbf{y}\in\textbf{L}^1. 
\end{array}
\]
\end{lem}
\begin{proof}
Let us start by showing that
\begin{equation}
\label{Lipschitz}
\begin{array}{cc}
|\rho(x) − \rho(y)| \leq\Vert x − y |\F \Vert_\infty, \textnormal{ for }x,y\in L^\infty_\F(\EE). 
\end{array}
\end{equation}
Indeed, if $x,y\in L^\infty_\F(\mathcal{E})$, it is clear that $x\leq y+\Vert x-y|\F\Vert_\infty$. Then, by monotonicity and cash-invariance it follows that $\rho(y)-\rho(x)\leq\Vert x-y|\F\Vert_\infty$. By reversing the roles of $x$ and $y$, we obtain (\ref{Lipschitz}).

Now, for $a\in\F$, (\ref{Lipschitz}) yields that 
\[
|1_a \rho(x)-1_a \rho(1_a x)|\leq 1_a\Vert x- 1_a x|\F\Vert_\infty=0.
\]
This means that $\rho$ is stable. Therefore, it generates a conditional function $\boldsymbol\rho:\textbf{L}^\infty\rightarrow \textbf{R}$, which clearly  satisfies 1, 2, and 4. 

Let us prove 3. Indeed, fix $\textbf{x},\textbf{y}\in\textbf{L}^\infty$ and $\textbf{r}\in\textbf{R}$ with $\textbf{0}\leq\textbf{r}\leq\textbf{1}$. Suppose first that $\textbf{r}\in \textbf{s}(\R)$. In this case, the result follows because $\rho$ is convex and stable. Now, suppose that $\textbf{r}$ is arbitrary. Let $H:=\left\{h_1,h_2,...\right\}$ be a countable dense subset of $\R$. For a fixed $n\in\N$, let us define $a_1:=(r\leq h_1<r+\frac{1}{n})$ and $a_k:=(r\leq h_k<r+\frac{1}{n})-\wedge_{i=1}^{k-1} a_i$ for $k>1$. Then, $\{a_k\}\in p(1)$, and we can define a $r_n:=\left(\sum_{k\in\N} h_k|a_k\right)\wedge 1$. Now, for arbitrary $\textbf{n}\in\textbf{N}$ of the form $n=\sum n_i|a_i$ let us put $r_n:=\sum_{i\in\N} r_i|a_i$. Then $\{\textbf{r}_\textbf{n}\}$ is a conditional sequence in $\textbf{s}(\R)$  with $\textbf{0}\leq\textbf{r}_\textbf{n}\leq\textbf{1}$, thus 
\[
\boldsymbol\rho(\textbf{r}_\textbf{n}\textbf{x}+(\textbf{1}-\textbf{r}_\textbf{n})\textbf{y})\leq \textbf{r}_\textbf{n}\boldsymbol\rho(\textbf{x})+(\textbf{1}-\textbf{r}_\textbf{n})\boldsymbol\rho(\textbf{y}).
\]

From 4 we know that $\boldsymbol\rho$ is conditionally continuous. Then, since $\nlim_\textbf{n} \textbf{r}_\textbf{n}=\textbf{r}$, we obtain the result by taking conditional limits.

Finally, let us note that  $\rho^*:L^1_\F(\mathcal{E})\rightarrow L^0(\F)$ is also stable (because the conditional expectation is). Therefore, it defines a conditional function $\boldsymbol\rho^*:\textbf{L}^1\rightarrow \textbf{R}$, and 
\[
\begin{array}{cc}
\boldsymbol{\rho}^*(\textbf{y}):=\nsup\left[\E_\PP[\textbf{x} \textbf{y}|\F]-\boldsymbol{\rho}(\textbf{x}) \:;\: \textbf{x}\in\textbf{L}^\infty \right] & \textnormal{ for }\textbf{y}\in\textbf{L}^1, 
\end{array}
\]
because the conditional order of $\textbf{R}$ is defined in terms of the partial order of $L^0(\F)$ (see remarks after \cite[Definition 2.15]{key-7}).
\end{proof}

A conditional function $\boldsymbol\rho:\textbf{L}^\infty\rightarrow \textbf{R}$ fulfilling the  conditions 1, 2 and 3 above is a \textit{conditional convex risk measure}. In fact, it is the natural extension to the setting of conditional sets of the classical notion of convex risk measure. 

%

Two well-known notions for static convex risk measures are the so-called Fatou and Lebesgue properties. Namely, a (static) convex risk measure $\rho:L^\infty\rightarrow\R$ is said to have the Fatou property if for every bounded sequence $\{x_n\}$ in $L^\infty$ which converges a.s to $x$ it follows that $\rho(x)\leq \liminf \rho(x_n)$. If it also happens that $\lim \rho(x_n)=\rho(x)$, then $\rho$ has the Lebesgue property.

We now want to introduce a conditional version for these properties. For this purpose, let us define a conditionally partial order in $\textbf{L}^\infty$. Namely, for given $\textbf{x},\textbf{y}\in\textbf{L}^\infty$ we write $\textbf{x}\leq \textbf{y}$ conditionally almost surely if $x\leq y$ almost surely. This conditional order allows to define the  following notions:

\begin{itemize}

	\item If $\{\textbf{x}_\textbf{n}\}$ is a conditional sequence in $\textbf{L}^\infty$, we define
		\[
	\begin{array}{cc}
	\underset{\textbf{n}}\nliminf \textbf{x}_\textbf{n}=\underset{\textbf{m}\in\textbf{N}}\nsup\:\underset{\textbf{n}\geq \textbf{m}}\ninf \textbf{x}_\textbf{n} & (\textnormal{resp. }\underset{\textbf{n}}\nlimsup \textbf{x}_\textbf{n}=\underset{\textbf{m}\in\textbf{N}}\ninf \:\underset{\textbf{n}\geq \textbf{m}}\nsup \textbf{x}_\textbf{n}).
	\end{array}
	\]
	\item A conditional sequence $\{\textbf{x}_\textbf{n}\}$ in $\textbf{L}^\infty$ is said to \textit{conditionally converge  almost surely} to $\textbf{x}$ if $\underset{\textbf{n}}\nliminf \textbf{x}_\textbf{n}=\underset{\textbf{n}}\nlimsup \textbf{x}_\textbf{n}$.
	\item A conditional convex risk measure $\boldsymbol\rho : \textbf{L}^\infty\rightarrow\textbf{R}$ is said to have the \textit{conditional Fatou property} if for every conditionally bounded sequence $\{\textbf{x}_\textbf{n}\}$ in $\textbf{L}^\infty$ which conditionally  converges a. s. to $\textbf{x}$ it holds that $\boldsymbol\rho(\textbf{x})\leq \underset{\textbf{n}}\nliminf\boldsymbol\rho(\textbf{x}_\textbf{n})$.   
	\item A conditional convex risk measure $\boldsymbol\rho : \textbf{L}^\infty\rightarrow\textbf{R}$ is said to have the \textit{conditional Lebesgue property} if for every conditionally bounded sequence $\{\textbf{x}_\textbf{n}\}$ in $\textbf{L}^\infty$ which conditionally  converges a. s. to $\textbf{x}$ it holds that $\nlim_{\textbf{n}}\boldsymbol\rho(\textbf{x}_\textbf{n})=\boldsymbol\rho(\textbf{x})$. 
\end{itemize}

\begin{rem}
\label{rem: ASconvergence}
Note that no conditional topology is involved for the conditional almost sure convergence defined above.

Let $\{x_n\}$ be a sequence in $L^\infty_\F(\EE)$, and let us construct from it the conditional sequence $\{\textbf{x}_\textbf{n}\}$. Then, inspection shows that $\underset{\textbf{n}}\nliminf \textbf{x}_\textbf{n}=(\underset{n}\essliminf x_n)|1$ and $\underset{\textbf{n}}\nlimsup \textbf{x}_\textbf{n}=(\underset{n}\esslimsup x_n)|1$. This means that $\{x_n\}$ converges a.s. to $x$ if, and only if, $\{\textbf{x}_\textbf{n}\}$  conditionally  converges a. s. to $\textbf{x}$.
\end{rem}

We have the following result:

\begin{prop}
\label{prop: equivFatou}
Let $\rho:L^\infty_\F(\EE)\rightarrow L^0(\F)$ be a convex risk measure, then the following are equivalent:

\begin{enumerate}
	\item $\rho|_{L^\infty(\EE)}:L^\infty(\EE) \rightarrow L^0(\F)$ has the Fatou property, i.e. for every bounded sequence $\{x_n\}\subset L^\infty(\EE)$ such that $\lim x_n=x$ a.s. it holds that $\rho(x)\leq \underset{n}\essliminf \rho(x_n)$ (resp. Lebesgue property, i.e. for every bounded sequence $\{x_n\}\subset L^\infty(\EE)$ such that $\lim x_n=x$ a.s. it holds that $\rho(x)=\underset{n}\lim\rho(x_n)$); 
	\item $\boldsymbol\rho : \textbf{L}^\infty\rightarrow\textbf{R}$ has the conditional Fatou property (resp. conditional Lebesgue property).
\end{enumerate}
\end{prop}
\begin{proof}

Let us give the proof for the Fatou property. The Lebesgue property is similar.

$2\Rightarrow 1:$ Let $\{x_n\}$ be a bounded sequence of $L^\infty(\EE)$, which converges a.s. to $x$. It defines a conditional sequence $\{\textbf{x}_\textbf{n}\}$ in $\textbf{L}^\infty$ which is conditionally bounded and conditionally converges a.s. to $\textbf{x}$. Then, having $\boldsymbol\rho$ the conditional Fatou property it happens that $\boldsymbol\rho(\textbf{x})\leq \underset{\textbf{n}}\nliminf\boldsymbol\rho(\textbf{x}_\textbf{n})$. But, it means that $\rho(x)\leq \underset{n}\essliminf \rho(x_n)$.

$1\Rightarrow 2:$ Let $\{\textbf{x}_\textbf{n}\}$ be a conditionally bounded sequence which conditionally converges a.s. to $\textbf{x}$. Then, we have that $x_n \rightarrow x$ a.s. and there exists $\eta\in L^0_{+}(\mathcal{F})$ with $|x_n|\leq \eta$ for every $n\in\N$. 

For each $k\in\N$, let us define $a_k:=(k-1\leq \eta <k)$. Then $\{a_k\}\in p(1)$ and $|1_{a_k} x_n|\leq 1_{a_k}\eta \leq k$ for every $k\in\N$. Therefore, since $\rho_0$ has the Fatou property, we see that $1_{a_k} \rho(1_{a_k} x)\leq 1_{a_k} \essliminf \rho(1_{a_k} x_n)$. Besides, $\rho$ is stable, hence $1_{a_k} \rho(x)\leq 1_{a_k} \essliminf \rho (x_n)$, which yields that $\rho(x)\leq \essliminf\rho(x_n)$.
\end{proof}





We have the following result:

\begin{thm}
\label{thm: Fatou}
Suppose that $\rho : L^\infty_{\mathcal{F}}(\mathcal{E})\rightarrow L^0(\F)$ is a convex risk measure. Then the following conditions are equivalent:
\begin{enumerate}
	\item $\rho$ can be represented by the Fenchel conjugate $\rho^*$, i.e, for $x\in {L}^\infty_{\mathcal{F}}(\mathcal{E})$
	\begin{equation}
	\label{fatou1}
	\rho(x)=\esssup \left\{ \E_\PP\left[x{y}\mid\mathcal{F}\right] - \rho^*({y}) \:;\: {y}\in {L}^1_{\mathcal{F}}(\mathcal{E}),y\leq 0, \E\left[{y}\mid\mathcal{F}\right]=-1\right\};
\end{equation}

	\item ${\rho |}_{L^\infty (\mathcal{E})}$ has the Fatou property;
	
	\item $\boldsymbol\rho$ has the conditional Fatou property;
	
	\item the level set $V_{c}(\rho):=\left\{x\in L^\infty_\F (\mathcal{E}) \:;\: \rho(x)\leq c \right\}$ is $\sigma (L^\infty_{\mathcal{F}}(\mathcal{E}),L^1_{\mathcal{F}}(\mathcal{E}))$-closed;

	\item the conditional level set $\textbf{V}_{\textbf{c}}(\boldsymbol\rho):=\left[\textbf{x}\in \textbf{L}^\infty \:;\: \boldsymbol\rho(\textbf{x})\leq \textbf{c} \right]$ is conditionally $\sigma (\textbf{L}^\infty,\textbf{L}^1)$-closed. 
	
\end{enumerate}
\end{thm}
\begin{proof}

We follow the line $1\Leftrightarrow 2 \Leftrightarrow 3 \Rightarrow 4 \Rightarrow 5 \Rightarrow 3$.

$1\Leftrightarrow 2$: It is an easy consequence of \cite[Theorem 1]{key-20} and \cite[Theorem 4.4]{key-37}, by just noting that we can suppose w.l.g. that $\rho(0)=0$ and, in view of 4 of Lemma \ref{lem: cont}, we have $\rho(L^\infty(\EE))\subset L^\infty(\F)$.

$2\Leftrightarrow 3$: it is just Proposition \ref{prop: equivFatou}.

%

$1\Rightarrow 4$: it is clear.

$4\Rightarrow 5$: The topology induced by $\sigma(L^\infty_{\mathcal{F}}(\mathcal{E}),L^1_{\mathcal{F}}(\mathcal{E}))_{cc}$ is finer than the topology induced by $\sigma(L^\infty_{\mathcal{F}}(\mathcal{E}),L^1_{\mathcal{F}}(\mathcal{E}))$. Theorem \ref{thm: conection} yields the result.

$5\Rightarrow 3$: Let $\{\textbf{x}_\textbf{n}\}$ be a conditionally bounded sequence which conditionally converges a.s. to $\textbf{x}$. Let us define $\textbf{a}:=\underset{\textbf{n}}\nliminf\boldsymbol\rho(\textbf{x}_\textbf{n})$. Due to the conditional version of Eberlein-\v{S}mulian theorem (see \cite[Theorem 4.7]{key-38}) and by taking a conditional subsequence, if necessary, we can suppose that $\boldsymbol\rho(\textbf{x}_\textbf{n})$ conditionally converges to $\textbf{a}$. Further, by cash invariance, we know that the conditional sequence $\{\textbf{z}_\textbf{n}\}$, with $\textbf{z}_\textbf{n}:=\boldsymbol\rho(\textbf{x}_\textbf{n})-\textbf{x}_\textbf{n}$, is a conditional sequence in $\textbf{V}_0(\boldsymbol\rho)$. Therefore, since $\textbf{V}_0(\boldsymbol\rho)$ is conditionally $\sigma (\textbf{L}^\infty,\textbf{L}^1)$-closed, by cash invariance, it suffices to show that $\{\textbf{z}_\textbf{n}\}$ conditionally weakly-$*$ converges to $\textbf{z}:=\textbf{a}-\textbf{x}$. Indeed, let us take $\textbf{x}_1 \in \textbf{L}^1$ and let us show that $\nlim_{\textbf{n}}\E_\PP[\textbf{x}_1 \textbf{z}_\textbf{n}|\F]=\textbf{z}$. We can suppose $\textbf{x}_1\geq \textbf{0}$ and $\E_\PP[x_1|\F]=1$ w.l.g. This allows us to choose a probability measure $Q\in\mathcal{P}_\F$ with $x_1 =\frac{d Q}{d\PP}$.

Now, in view of remark \ref{rem: ASconvergence}, we know that $\{z_n\}$ converges a.s. to $z$. By the theorem of dominated convergence for conditional expectations, we obtain $\E_Q[z_n|\F]$ converges a.s. to $z$. Consequently, the conditional sequence $\E_Q[\textbf{z}_\textbf{n}|\F]$ conditionally converges to $\textbf{z}$, and the result follows. 
\end{proof}

%


Below, we state the extension of the Jouini-Schachermayer-Touzi theorem:

\begin{thm}
\label{th: LebesgueProperty}
Suppose $\rho : L^\infty_{\mathcal{F}}(\mathcal{E})\rightarrow L^0(\F)$ is a convex risk measure such that $\rho|_{L^\infty(\EE)}$ has the Fatou property, and put:
\[
\begin{array}{cc}
\rho_0^*(y):=\esssup\{\E_\PP[x y|\F]-\rho_0(x) \:;\: x\in L^\infty(\EE) \} & \textnormal{ for }y\in L^1(\EE)
\end{array}
\]
 with $\rho_0:=\rho|_{L^\infty_\F(\EE)}$. The following statements are equivalent: 
\begin{enumerate} 
\item The conditional set $\textbf{V}_{\boldsymbol\rho^*}(\textbf{c}):=\left[\textbf{y}\in \textbf{L}^1 \:;\: \boldsymbol\rho^*(\textbf{y})\leq \textbf{c}\right]$ is  conditionally $\sigma\left(\textbf{L}^1,\textbf{L}^\infty\right)$-compact for all $\textbf{c}\in \textbf{R}$;

\item $\rho_0$ has the Lebesgue property;

\item $\boldsymbol\rho$ has the conditional Lebesgue property;

\item for every $x\in L^\infty(\EE)$ there is $y\in L^1(\EE)$ with $y\leq 0$ and $\E_\PP[y|\F]=-1$ such that $\rho(x)=\E_\PP[x y|\F]-\rho_0^*(y)$;
\item For every $x\in L^\infty_\F(\EE)$ there is $y\in L^1_\F(\EE)$ with $y\leq 0$ and $\E_\PP[y|\F]=-1$ such that $\rho(x)=\E_\PP[x y|\F]-\rho^*(y)$. 
\end{enumerate}
\end{thm}

Before proving the main result, we need some preliminary results.  

\begin{prop}
\label{prop: condSimpleFunc}
The conditional set $\textbf{L}^\infty$ is conditionally dense in $\textbf{L}^1$ with respect the conditional norm topology.
\end{prop}
\begin{proof}
Indeed, let us choose $\textbf{x}=x|1\in\textbf{L}^1$. Since $x=x^+-x^-$, we can assume w.l.g. that $x\geq 0$. Let us choose some sequence $\{x_n\}$ in $L^\infty(\EE)$ such that $x_n\nearrow x$ a.s. Then, by monotone convergence we have $\E_\PP[x_n-x|\F]\nearrow 0$. If we construct the conditional sequence $\{\textbf{x}_\textbf{n}\}\sqsubset \textbf{L}^\infty$, it follows that $\nlim_\textbf{n}\nVert \textbf{x}_\textbf{n}-\textbf{x}\nVert_1=\textbf{0}$.
\end{proof}

\begin{lem}
\label{lem: WCG}
The conditional unit ball of $\textbf{L}^\infty$ is conditionally weakly-$*$ sequentially compact
\end{lem}

\begin{proof}
In view of Proposition \ref{prop: condSimpleFunc}, we have that $\textbf{L}^2$ is conditionally dense in $\textbf{L}^1$. This mean that $\textbf{L}^1=\overline{\textbf{L}^2}=\overline{\spa \textbf{B}_{\textbf{L}^2}}$, where $\textbf{B}_{\textbf{L}^2}$ is the conditional unit ball of $\textbf{L}^2$. By the conditional version of Amir-Lindenstrauss (see \cite[Theorem 4.8]{key-38}), it follows that the conditional unit ball of $(\textbf{L}^1)^*=\textbf{L}^\infty$ is conditionally weakly-$*$ sequentially compact.  
\end{proof}

A similar result is proved in Lemma 2 of \cite{key-39}:

\begin{lem}
\label{lem: seq}
Let $\{y_n\}$ be a sequence in $L^0(\F)$ such that  $\esslimsup_n y_n=y$. Suppose that, for any $n\in L^0(\N)$ of the form $n=\sum_k 1_{a_k} n_k$, we define $y_n:=\sum_k 1_{a_k} y_{n_k}$. Then, there is a sequence $n_1<n_2<...$ in $L^0(\N)$, such that the sequence $\{y_{n_m}\}$ converges a.s to $y$.   
\end{lem}

The following results is a plain adaptation of results from \cite{key-37} to the present setting (see Lemma 2.16, Proposition 3.3 and Proposition 3.5).

\begin{lem}
\label{lem: Guo}
The following properties hold:
\begin{enumerate}
	\item $s(L^\infty(\EE))=L^\infty_\F(\EE)$.
	\item Let $\boldsymbol\rho:\textbf{L}^\infty\rightarrow\textbf{R}$ be a conditional convex risk measure. Then, for $y\in L^1(\EE)$
	\[
	\rho^*(y)=\esssup\left\{\E_\PP[x y|\F]-\rho(x) \:;\:x\in L^\infty(\EE) \right\}.
	\]
\end{enumerate}
\end{lem}
	%
%
%

Now, let us turn to prove the main theorem.

\begin{proof}
We will follow the line $3\Leftrightarrow 2\Leftrightarrow 4\Leftrightarrow 5\Leftrightarrow  1$.


$3\Leftrightarrow 2$ is just Proposition \ref{prop: equivFatou}.

$2\Rightarrow 4$: We shall use a method  similar to the scalarization technique used by  Deflefsen and Scandolo in the proof of  $a\Rightarrow c$ in \cite[Theorem 1]{key-20}. First, we can assume, by applying a translation if necessary, that $\rho(0)=0$. If so, for $x\in L^\infty (\mathcal{E})$, due to 4 of Lemma \ref{lem: cont} we have that $|\rho(x)|\leq\Vert x|\F\Vert_\infty\leq \Vert x\Vert_\infty$, hence  $\rho(L^\infty (\mathcal{E}))\subset L^\infty (\mathcal{F})$.

 Let us fix $x\in L^\infty(\EE)$ and put $\rho_0=\rho|_{L^\infty(\EE)}$. Since $\rho(x)\geq \E_\PP[x y|\F]-\rho^*_0(y)$ for all $y\in L^1(\EE)$, it suffices to show that there exists $y\in L^1(\EE)$, with $y\leq 0$ and $\E_\PP[y|\F]=-1$, such that 
\begin{equation}
\label{eq: prepLeb}
\E_\PP[\rho(x)] =\E_\PP [\E_\PP[x y|\F]-\rho^*_0(y)].
\end{equation}

Let us take $\rho':L^\infty(\EE)\rightarrow\R$, where $\rho'(x):=\E_\PP[\rho(x)]$ for each $x\in L^\infty(\EE)$, which is a convex risk measure. Moreover, we claim that $\rho'$ has the Lebesgue property. Indeed, given a bounded sequence $\{x_n\}\subset L^\infty(\EE)$ which converges  a.s. to $x$, having $\rho_0$ the Lebesgue property, it holds that $\lim_n \rho_0(x_n)=\rho_0(x)$. Further, due to 4 of Lemma  of \ref{lem: cont}, we have that $\{\rho_0(x_n)\}$ is bounded. Thus, by dominated convergence, we conclude that  $\lim_n \rho'(x_n)=\rho'(x)$. Now, by the original Jouini-Schachermayer-Touzi Theorem, it follows that $\rho'(x)=\E_\PP[x y] - (\rho')^*(y)$ for some $y\in L^1(\EE)$ with $y\leq 0$ and $\E_\PP[y]=-1$, where $(\rho')^*(y)=\sup\left\{\E_\PP[x y]-\rho'(x) \:;\: x\in L^\infty(\EE)\right\}$. 
 
Furthermore, we claim that $\E_\PP[y|\F]=-1$. Arguing by way of contradiction, suppose for instance that $a:=(\E_\PP[y|\F]>-1)$ has positive probability. If so, for fixed $\lambda\in\R^+$, we have 
\[
(\rho')^*(y)\geq \E_\PP[1_a \lambda y] - \rho'(\lambda 1_a)=\lambda \left( \E_\PP[1_a y] + \PP(a)\right)=\lambda \E_\PP\left[ 1_a \left(\E_\PP[ y|\mathcal{F}] + 1\right)\right].
\]
But this is impossible, since $\lambda$ is arbitrarily large and $(\rho')^*(y)<+\infty$. 

In addition, by using the same techniques as in \cite[Theorem 1]{key-20}, we can obtain that $\E_\PP[\rho_0^*(y)]=(\rho')^*(y)$. Therefore, we finally have 
\[
\E_\PP[\E_\PP[x y|\F]- \rho_0^*(y)]=\E_\PP[x y]-(\rho')^*(y)=\rho'(y)=\E_\PP[\rho(x)].
\]

$4\Rightarrow 2$: We will use the same reduction trick again. Let $\{x_n\}$ be a bounded sequence in $L^\infty(\EE)$ such that $x_n$ converges a.s. to $x$.  Since $\rho$ has the Fatou property, we have that $\rho(x)\leq \underset{n} \essliminf \rho (x_n)$. It suffices to show that $\rho(x)\geq \underset{n} \esslimsup \rho (x_n)$. Let us argue to get contradiction. Suppose that there exists $a\in\A$, $a\neq 0$ such that  $\rho(x)< \underset{n} \esslimsup \rho (x_n)$ on $a$. We can assume $a=1$ w.l.g. Then,  let us consider $\rho'(x):=\E_\PP[\rho(x)]$ for $x\in L^\infty(\EE)$. Thereby, we have
\[
\rho'(x)< \E_\PP[\underset{n} \esslimsup \rho (x_n)].
\]

Due to Lemma \ref{lem: seq}, we can construct a bounded sequence $\{z_n\}$  in a such a way that $\lim_n \rho (z_n)=\underset{n} \esslimsup \rho (x_n)$, and which converges a.s  to $x$. Then, by dominated convergence, we obtain
\begin{equation}
\label{eq: lebIn}
\lim_n\E_\PP[\rho (z_n)]=\E_\PP[\lim_n\rho (z_n)]=\E_\PP[\underset{n} \esslimsup \rho (x_n)]>\rho'(x).
\end{equation} 

But, on the other hand, by assumption we have that, for each $x\in L^\infty(\EE)$, there is $y\in L^1(\EE)$ with $y\leq 0$ and $\E_\PP[y|\F]=-1$, such that $\rho'(x)=\E_\PP[\E_\PP[x y|\F]-\rho_0^*(y)]=\E_\PP[x y] - \E_\PP[\rho_0^*(y)]=\E_\PP[x y] - (\rho')^*(y)$. Also, by dominated convergence, we obtain  $\rho'$ has the Fatou property. In view of the original Jouini-Schachermayer-Touzi Theorem, it follows that $\rho'$ has the Lebesgue property. Hence, $\lim_n\E_\PP[\rho (z_n)]=\rho'(x)$, which is a contradiction, in view of (\ref{eq: lebIn}).

$4\Rightarrow 5:$ By 2 of Lemma \ref{lem: Guo}, it holds that $\rho^*(y)=\rho_0^*(y)$ for all $y\in L^1(\EE)$. Now, let us fix $x\in L^\infty_\F(\EE)$. Due to 1 of Lemma \ref{lem: Guo}, with have that $x=\sum 1_{a_k} x_k$ with $x_k\in L^\infty(\EE)$ and $\{a_k\}\in p(1)$. By assumption, we can choose $y\in L^1(\EE)$ satisfying 4. Thereby, by using the fact that $\rho$ is stable, we have
\[
\rho(x)=\sum 1_{a_k}\rho(x_k)=\sum 1_{a_k}\left( \E_\PP[x_k y|\F] - \rho_0^*(y) \right)=\E_\PP[x y|\F] - \rho^*(y).
\]
$5\Rightarrow 4:$  For given $x\in L^\infty(\EE)$, by assumption there exists $y\in L^1_\F(\EE)$ satisfying 5, which is of the form $y=\xi_0 y_1$ with $\xi_0\in L^0(\F)$ and $y_1\in L^1(\EE)$. It follows that $-1=\E_\PP[y|\F]=\xi_0$, hence $y\in L^1(\EE)$.

$5\Rightarrow 1$: Due to Lemma \ref{lem: WCG}, we know that the conditional unit ball of $\textbf{L}^\infty$ is conditionally weakly-$*$ sequentially compact, in particular, it is conditionally weakly convex block compact. Then, Theorem \ref{thm: James} tells us that $\textbf{V}_{\boldsymbol\rho^*}(\textbf{c})$ is conditionally weakly compact.

$1\Rightarrow 5$: Let us fix $\textbf{x}\in\textbf{L}^\infty$. Since $\rho|_{L^\infty(\EE)}$ has the Fatou property, due to Theorem \ref{thm: Fatou}, we have that 
\[
\boldsymbol\rho(\textbf{x})=\nsup \left[\E_\PP[\textbf{x}\textbf{y}|\F]-\boldsymbol\rho^*(\textbf{y})\:;\: \textbf{y}\leq \textbf{0}, \E_\PP[\textbf{y}|\F]=-\textbf{1}\right].
\]
Thus, we can take a conditional sequence $\{\textbf{y}_\textbf{n}\}_\textbf{n}$ in $\textbf{L}^1$ with $\textbf{y}_\textbf{n}\leq \textbf{0}$  and $\E\left[\textbf{y}_\textbf{n}\mid\mathcal{F}\right]=-\textbf{1}$ for each $\textbf{n}\in \textbf{N}$, such that $\boldsymbol\rho(\textbf{x})=\nlim(\E_\PP[\textbf{x}\textbf{y}_\textbf{n}|\F]-\boldsymbol\rho^*(\textbf{y}_\textbf{n}))$. This means that the conditional sequence $\{\boldsymbol\rho^*(\textbf{y}_\textbf{n})\}$ is conditionally bounded, and  we therefore have $\{\textbf{y}_\textbf{n}\}\sqsubset \textbf{V}_{\boldsymbol\rho^*}(\textbf{c})$, for some $\textbf{c}\in \textbf{R}^+$. Then, the conditional weak compactness and the conditional Eberlein-\v{S}mulian theorem (see \cite[Theorem 4.7]{key-38}) allows us to suppose that $\{\textbf{y}_\textbf{n}\}$ conditionally weakly converges to some $\textbf{y}\in \textbf{L}^1$ with $\textbf{y}\leq \textbf{0}$  and $\E\left[\textbf{y}\mid\mathcal{F}\right]=-\textbf{1}$. Consequently, we have $\nliminf_\textbf{n}\boldsymbol\rho^*(\textbf{y}_\textbf{n})\geq \boldsymbol\rho^*(\textbf{y})$. Thus, we finally have $\boldsymbol\rho(\textbf{x})=\nlim(\E_\PP[\textbf{x}\textbf{y}_\textbf{n}|\F]-\boldsymbol\rho^*(\textbf{y}_\textbf{n}))=\E_\PP[\textbf{x}\textbf{y}|\F]-\nliminf_\textbf{n}\boldsymbol\rho^*(\textbf{y}_\textbf{n})\leq \E_\PP[\textbf{x}\textbf{y}|\F]-\boldsymbol\rho^*(\textbf{y})$, which means that $\boldsymbol\rho(\textbf{x})=\E_\PP[\textbf{x}\textbf{y}|\F]-\boldsymbol\rho^*(\textbf{x})$.  
\end{proof}

\begin{appendix}
\section{Appendix}

Let us  prove Theorem \ref{thm: EberleinSmulianII}. First, we need to prove some preliminary results:

\begin{lem}
\label{lem: finiteSet} 
Let $\{\textbf{x}_\textbf{n}\}$ be a conditional sequence in a conditional set $\textbf{E}$. Suppose that the conditional subset $[\textbf{x}_\textbf{n}\:;\:\textbf{n}\in\textbf{N}]$ is conditionally finite, then there exists some $\textbf{x}\in\textbf{E}$  such that 
\[
\textnormal{for every }\textbf{m}\in\textbf{N}\textnormal{ there is }\textbf{n}\geq\textbf{m}\textnormal{ such that }\textbf{x}_\textbf{n}=\textbf{x}.
\]
\end{lem}
\begin{proof}
Let us put $[\textbf{x}_\textbf{n}\:;\:\textbf{n}\in\textbf{N}]=[\textbf{z}_\textbf{n}\:;\:\textbf{1}\leq\textbf{k}\leq\textbf{m}]$, and  for each $\textbf{k}\in\textbf{N}$ with $\textbf{1}\leq\textbf{k}\leq\textbf{m}$, let us define
\[
\textbf{l}_\textbf{k}:=\nmax[\textbf{n}\in\textbf{N}\:;\:\textbf{x}_\textbf{n}=\textbf{z}_\textbf{k}]\in\textbf{N}\sqcup [+\infty].
\]
We claim that $\textbf{l}:=\underset{\textbf{1}\leq\textbf{k}\leq\textbf{m}}\nmax\textbf{l}_\textbf{k}=+\infty$. Indeed, let us suppose $l|a\in\textbf{N}$ for some $a\in\A$, $a\neq 0$. We can assume $a=1$ w.l.g. Then, for $\textbf{p}>\textbf{l}$ we have that $\textbf{x}_\textbf{p}\in[\textbf{z}_\textbf{k}\:;\:\textbf{1}\leq\textbf{k}\leq\textbf{m}]$; and therefore, $\textbf{x}_\textbf{p}=\textbf{z}_\textbf{k}$ for some $\textbf{1}\leq\textbf{k}\leq\textbf{m}$. But necessarily $\textbf{p}\leq\textbf{l}_\textbf{k}\leq\textbf{l}$, a contradiction. 

Further, since the conditional maximum of a conditionally finite set is attained, we have $\textbf{l}_\textbf{k}=\textbf{l}=\infty$ for some $\textbf{1}\leq\textbf{k}\leq\textbf{m}$. It suffices to take $\textbf{x}:=\textbf{z}_\textbf{k}$.
\end{proof}

\begin{prop}
\label{prop: clusterPoint}
Let $(\textbf{X},\mathcal{T})$ be a conditionally compact topological space, then every conditional sequence has a conditional cluster point.
\end{prop}
\begin{proof}
Let $\{\textbf{x}_\textbf{n}\}$ be a conditional sequence in $\textbf{X}$. Let
\[
a:=\vee\left\{b\in\A\:;\:\{\textbf{x}_\textbf{n}\}|b\textnormal{ has a conditional cluster point in }\textbf{X}|b\right\}.
\]
Arguing by way of contradiction, let us suppose $a<1$. We can assume $a=0$ w.l.g. If so, due to Lemma \ref{lem: finiteSet}, $\{\textbf{x}_\textbf{n}\}|b$ is not conditionally finite for every $b\in\A$, $b\neq 0$. Then, we claim that, for each $\textbf{x}\in\textbf{E}$, there exists $\textbf{O}_\textbf{x}\in\mathcal{T}$ such that
\[
\begin{array}{cc}
\textbf{O}_\textbf{x}\sqcap[\textbf{x}_\textbf{n}]=\textbf{F}_\textbf{x}, & \textnormal{ where }\textbf{F}_\textbf{x}\textnormal{ is conditionally finite (not necessarily on $1$)}.
\end{array}
\] 
Indeed, let us fix $\textbf{x}\in \textbf{X}$ and let us put
\[
b:=\vee\left\{a\in\A\:;\: \exists\textbf{O}\in\mathcal{T},\:\textbf{x}\in\textbf{O},\: [\textbf{x}_\textbf{n}]|a\sqcap\textbf{O}|a\textnormal{ is conditionally finite}\right\}.
\]

Note that $b$ is attained. Besides, we have that $b=1$; otherwise,  we would have that for each $\textbf{O}\in\mathcal{T}$ with $\textbf{x}\in\textbf{O}$, $\nmax[\textbf{n}\:;\: \textbf{x}_\textbf{n}\in \textbf{O}]=+\infty$ on $b^c$, and $\textbf{x}|b^c$ would be a conditional cluster point of $[\textbf{x}_\textbf{n}]|b^c$. 

Now, for each conditionally finite subset $\textbf{F}$ of $[\textbf{x}_\textbf{n}]$ (not necessarily on $1$), including $[\textbf{x}_\textbf{n}]|0$, let
\[
\textbf{O}_\textbf{F}:=\nint\left([\textbf{x}_\textbf{n}]^\sqsubset\sqcup\textbf{F} \right).
\]
We have that $\textbf{O}_\textbf{x}\sqsubset\textbf{O}_{\textbf{F}_\textbf{x}}$, hence $[\textbf{O}_\textbf{F}\:;\:\textbf{F}\textnormal{ conditionally finite}]$ is a conditional open covering of $\textbf{X}$, which is conditionally compact. Thus, there exists a conditionally finite collection $[\textbf{F}_\textbf{k}\:;\:1\leq\textbf{k}\leq\textbf{m}]$ of conditionally finite subsets of $\textbf{N}$ such that
\[
\textbf{X}=\underset{1\leq\textbf{k}\leq\textbf{m}}\sqcup\textbf{O}_{\textbf{F}_\textbf{k}}
=\underset{1\leq\textbf{k}\leq\textbf{m}}\sqcup\left([\textbf{x}_\textbf{n}]^\sqsubset\sqcup\textbf{F}_\textbf{k}\right)
=[\textbf{x}_\textbf{n}]^\sqsubset\sqcup\underset{1\leq\textbf{k}\leq\textbf{m}}\sqcup\textbf{F}_\textbf{k}.
\]
But this means that
\[
[\textbf{x}_\textbf{n}]\sqsubset \underset{\textbf{1}\leq\textbf{k}\leq\textbf{m}}\sqcup\textbf{F}_\textbf{k},
\]
which is a contradiction, because $[\textbf{x}_\textbf{n}]$ is not conditionally finite.
\end{proof}

Finally, let us  prove Theorem \ref{thm: EberleinSmulianII}:

\begin{proof}
Let  $\textbf{x}^{**}\in\textbf{E}^{**}\sqcap\textbf{j}(\textbf{E})^\sqsubset$. The strategy of the proof of \cite[Theorem 4.7]{key-38} allows us to construct conditional sequences $\{\textbf{x}_\textbf{n}\}$ in $\textbf{K}$; $\{\textbf{x}_\textbf{n}^*\}$ in $\textbf{E}^*$; and $\{\textbf{n}_\textbf{k}\}$ in $\textbf{N}$, which is conditionally increasing, so that
\begin{equation}
\label{eq: EberleinI}
\begin{array}{cc}
\frac{\nVert\textbf{y}^{**}\nVert}{\textbf{2}}\leq\nsup[|\textbf{y}^{**}(\textbf{x}_\textbf{n}^*)|\:;\:\textbf{n}\in\textbf{N}], & \textnormal{ for all }\textbf{y}^{**}\in\overline{[\textbf{x}^{**},\textbf{x}^{**}-\textbf{x}_\textbf{n}\:;\:\textbf{n}\in\textbf{N}]}^{\nVert\cdot\nVert}
\end{array}
\end{equation}

\begin{equation}
\label{eq: EberleinII}
\begin{array}{cc}
|\textbf{x}^{**}(\textbf{x}^*_\textbf{n})-\textbf{x}^*_\textbf{n}(\textbf{x}_\textbf{k})|\leq\frac{\textbf{1}}{\textbf{k}} &
\textnormal{ for all }1\leq\textbf{n}\leq\textbf{n}_\textbf{k}.
\end{array}
\end{equation}

On the other hand, the conditional Banach-Alaoglu Theorem (see \cite[Theorem 5.10]{key-7}) yields that $\overline{\textbf{j}(\textbf{K})}^{\omega^*}$ is conditionally weakly-$*$ compact. In view of Proposition \ref{prop: clusterPoint}, $\{\textbf{j}(\textbf{x}_\textbf{n})\}$  has a conditional weak-$*$ cluster point $\textbf{x}^{**}_0\in\overline{\textbf{j}(\textbf{K})}^{\omega^*}$. Let us show that $\textbf{x}^{**}_0=\textbf{x}^{**}$.

Due to (\ref{eq: EberleinI}), it follows 
\begin{equation}
\label{eq: EberleinII}
\frac{\textbf{1}}{\textbf{2}}\nVert\textbf{x}^{**}-\textbf{x}^{**}_0\nVert\leq \nsup[|\textbf{x}^{**}(\textbf{x}^*_\textbf{n})-\textbf{x}^{**}_0(\textbf{x}^*_\textbf{n})|\:;\:\textbf{n}\in\textbf{N}].
\end{equation}

Now, for fixed $\textbf{n}\in\textbf{N}$, let us take $\textbf{p}$ with $\textbf{n}\leq\textbf{n}_\textbf{p}$ and fix some $\textbf{k}$ with $\textbf{p}\leq\textbf{k}$. It holds
\[
|\textbf{x}^{**}(\textbf{x}^*_\textbf{n})-\textbf{x}^{**}_0(\textbf{x}^*_\textbf{n})|\leq 
|\textbf{x}^{**}(\textbf{x}^*_\textbf{n})-\textbf{x}^*_\textbf{n}(\textbf{x}_\textbf{k})|+|\textbf{x}^*_\textbf{n}(\textbf{x}_\textbf{k})-\textbf{x}^{**}_0(\textbf{x}^*_\textbf{n})|\leq
\frac{\textbf{1}}{\textbf{p}}+|\textbf{x}^*_\textbf{n}(\textbf{x}_\textbf{k})-\textbf{x}^{**}_0(\textbf{x}^*_\textbf{n})|.
\]

Given that $\textbf{x}^{**}_0$ is a conditional $\omega^*$-cluster point of $\{\textbf{j}(\textbf{x}_\textbf{n})\}$, we can choose $\textbf{k}$ so that $|\textbf{x}^*_\textbf{n}(\textbf{x}_\textbf{k})-\textbf{x}^{**}_0(\textbf{x}^*_\textbf{n})|\leq \frac{\textbf{1}}{\textbf{p}}$. Finally, since $\textbf{p}\in\textbf{N}$ arbitrary, in view of (\ref{eq: EberleinII}), we conclude that $\textbf{x}^{**}_0=\textbf{x}^{**}$.

\end{proof}
\end{appendix}

\end{document}